\documentclass{amsart}

\usepackage{hyperref}

\usepackage{amsmath,amssymb,amsthm}
\usepackage{bbm,mathrsfs}

\newtheorem{theorem}{Theorem}[section]
\newtheorem{proposition}[theorem]{Proposition}

\newtheorem{lemma}[theorem]{Lemma}
\theoremstyle{definition}
\newtheorem{definition}[theorem]{Definition}
\newtheorem{example}[theorem]{Example}
\newtheorem{examples}[theorem]{Examples}
\theoremstyle{remark}
\newtheorem{remark}[theorem]{Remark}

\numberwithin{equation}{section}
\makeatletter
\renewcommand{\p@enumii}{\theenumi-}
\renewcommand{\p@enumiii}{\theenumi-\theenumii-}
\renewcommand{\p@enumiv}{\theenumi-\theenumii-\theenumiii-}
\makeatother

\newcommand{\hits}{\rightharpoonup}
\newcommand{\hitted}{\leftharpoonup}
\newcommand{\id}{\mathsf{id}}
\newcommand{\co}{\mathsf{co}}
\newcommand{\coo}{\mathsf{coinv}}
\newcommand{\ord}{\mathsf{ord}}
\renewcommand{\dim}{\mathsf{dim}}
\newcommand{\sgn}{\mathsf{sgn}}

\newcommand{\kk}{\Bbbk}
\newcommand{\G}{\mathbf{G}}
\newcommand{\Z}{\mathbb{Z}}
\newcommand{\sigmab}{\pmb{\sigma}}
\newcommand{\eb}{\pmb{e}}
\newcommand{\A}{\mathcal{A}}
\newcommand{\AM}{H}
\newcommand{\B}{\mathcal{B}}
\newcommand{\C}{\mathcal{C}}
\newcommand{\HH}{\mathcal{H}}
\newcommand{\KK}{\mathcal{K}}
\newcommand{\CC}{\mathsf{Spl}}
\newcommand{\T}{\mathscr{T}}
\newcommand{\PP}{\mathbf{P}}
\newcommand{\s}{\mathsf{s}}

\newcommand{\sym}{{\mathfrak{S}_3}}
\newcommand{\z}{\xi}

\newcommand{\Vect}{\mathcal{V}}
\newcommand{\sVect}{\mathcal{SV}}
\newcommand{\stw}{\mathsf{s}\text{-}\mathsf{sym}}

\newcommand{\YD}[1]{{}^{#1}_{#1}\mathcal{YD}}
\newcommand{\YDH}[1]{{}^{#1}_{#1}\mathcal{YDH}}
\newcommand{\lcomod}[1]{{}^{#1}\mathcal{M}}
\newcommand{\rcomod}[1]{\mathcal{M}^{#1}}
\newcommand{\lmod}[1]{{}_{#1}\mathcal{M}}

\newcommand{\lsmod}[1]{{}_{#1}\mathcal{SM}}
\newcommand{\rscomod}[1]{\mathcal{SM}^{#1}}
\newcommand{\AD}[1]{\mathsf{AD}(#1)}
\newcommand{\SD}[1]{\mathsf{SD}(#1)}

\newcommand{\rhmod}[1]{\mathcal{M}_{#1}^{#1}}
\newcommand{\Ker}{\mathsf{Ker}}

\newcommand{\ratimes}{\mathbin{>\hspace{-4.8pt}\lhd}}
\newcommand{\rctimes}{\mathbin{>\hspace{-2.5pt}\blacktriangleleft}}
\newcommand{\biprod}{\mathbin{\#}}

\newcommand{\HopfAuto}{\mathsf{Aut}_{\mathrm{Hopf}}}
\newcommand{\qedheree}{\hfill$\square$}

\newcommand{\tabb}[1]{\begin{tabular}{l}#1\end{tabular}}

\newcommand{\fd}{\mathsf{fd}}
\usepackage{tikz-cd}

\allowdisplaybreaks

\begin{document}

\title{On classification of Hopf superalgebras of low dimension}

\author{Taiki Shibata}
\address[T.~Shibata]{Department of Applied Mathematics, Okayama University of Science,
1-1 Ridai-cho Kita-ku Okayama-shi, Okayama 700-0005, JAPAN}
\email{shibata@ous.ac.jp}

\author{Ryota Wakao}
\address[R.~Wakao]{Department of Applied Mathematics, Okayama University of Science,
1-1 Ridai-cho Kita-ku Okayama-shi, Okayama 700-0005, JAPAN}
\email{r23nda8mr@ous.jp}

\date{}

\begin{abstract}
We examine the inverse procedure of the Radford-Majid bosonization for Hopf superalgebras and give a handy method for enumerating Hopf superalgebras whose bosonization is isomorphic to a given Hopf algebra.
As an application, we classify Hopf superalgebras of dimension up to $5$ and give examples of higher dimensions.
\end{abstract}

\maketitle

\setcounter{tocdepth}{2}
\tableofcontents

\section{Introduction}
The classification problem of finite-dimensional Hopf algebras has been actively studied by many researchers since it was proposed by Kaplansky in 1975; see the survey \cite{BeaGar13}.
Due to the importance of supersymmetry in mathematics and mathematical physics, we are interested in the classification problem of finite-dimensional {\em Hopf superalgebras} over an algebraically closed field $\kk$ of characteristic zero.
Among them, super-(co)commutative ones have been fully classified by the following result:
According to Kostant~\cite{Kos77} (see also \cite{AndEtiGel01}), a finite-dimensional super-cocommutative Hopf superalgebra can be decomposed into a semidirect product of the group algebra $\kk\Gamma$ of a finite group $\Gamma$ and the exterior superalgebra $\bigwedge(V)$ over a finite-dimensional $\kk\Gamma$-module $V$.

In 2014, Aissaoui and Makhlouf~\cite{AisMak14} studied the classification of Hopf superalgebras of dimension up to $4$ using computer software and found some non-trivial Hopf superalgebras (see Example~\ref{ex:AM}).
In this paper, we use an approach different from that of \cite{AisMak14}, classify Hopf superalgebras of dimension up to $5$, and give new examples.
The key ingredient of our approach is the {\it bosonization} technique introduced by Radford~\cite{Rad85} and Majid~\cite{Maj94} (see Section~\ref{subsec:boson}).
This technique involves reducing Hopf superalgebras $\HH$ to ordinary Hopf algebras $\widehat{\HH}$, which allows us to apply various (classification) results obtained in ordinary settings.

Let us explain our approach in detail.
Let $A$ be a Hopf algebra over $\kk$.
It is known that there is a one-to-one correspondence between the isomorphism classes of pairs of Hopf algebra maps $\iota:\kk\Z_2\to A$ and $\pi:A\to\kk\Z_2$ with $\pi\circ\iota=\id$ and the isomorphism classes of Hopf algebras $\HH$ in the category of $\Z_2$-Yetter-Drinfeld modules $\YD{\kk\Z_2}$ such that $\widehat{\HH}\cong A$.
We show that the set of all such pairs $(\iota,\pi)$ is parameterized by the set
\[
\AD{A} = \{ (g, \alpha) \in \G(A) \times \G(A^\circ) \mid \ord(g) = 2, \ord(\alpha) = 2,  \alpha(g) = -1 \},
\]
whose elements are called {\it admissible data} for $A$ in this paper (Definition~\ref{def:yddata}).
Here, for a Hopf algebra $K$, we have denoted by $\G(K)$ (resp.~$K^\circ$) the group of group-like elements of $K$ (resp.~the finite dual of $K$).
Given an admissible datum $(g,\alpha)$, let $(\iota_{(g,\alpha)},\pi_{(g,\alpha)})$ be the corresponding pair described in Section~\ref{subsec:YD-data}.
We show that the Hopf algebra $\HH$ in $\YD{\kk\Z_2}$ corresponding to this pair is a Hopf superalgebra whose $\bar1$-part is non-zero if and only if
\[
\alpha\hits a \hitted \alpha = g a g \quad (a \in A)
\quad \text{and} \quad
g\not\in Z(A),
\]
where $Z(A)$ is the center of $A$.
We define $\SD{A}$ to be the subset of $\AD{A}$ consisting of elements satisfying this condition and call its elements {\em super-data} for $A$ (Definition~\ref{def:SD}).
The group of Hopf algebra automorphisms of $A$ naturally acts on $\SD{A}$.
A key observation is that there is a one-to-one correspondence between the orbits of $\SD{A}$ and the isomorphism classes of the set of all Hopf superalgebras $\HH$ such that $\HH_{\bar1}\neq0$ and $\widehat{\HH} \cong A$.
Moreover, the Hopf superalgebra $\HH$ is semisimple (resp.~pointed) if and only if $A$ is semisimple (resp.~pointed).

In this paper, we study Hopf superalgebras by utilizing the above bijection.
One of our applications classifies finite-dimensional Hopf superalgebras of prime dimensions.
By the classification results of Hopf algebras of dimension $2p$ for a prime number $p$ by Masuoka~\cite{Mas95} and Ng~\cite{Ng05},
we immediately get the following result.

\begin{theorem}[= Theorems~\ref{prp:2-dim} and \ref{thm:main1}]
Up to isomorphism, $\bigwedge(\kk)$ is the only Hopf superalgebra of dimension $2$ whose $\bar{1}$-part is non-zero.
If $p$ is an odd prime number, then a Hopf superalgebra of dimension $p$ is purely even, that is, its $\bar{1}$-part is zero (thus such a Hopf superalgebra is isomorphic to $\kk\Z_p$ by Zhu~\cite{Zhu94}).
\end{theorem}

We fix a primitive fourth root $\zeta_4\in\kk$ of unity.
By the above theorem and the classification results of Hopf algebras of dimension $8$ obtained by Masuoka~\cite{Mas95b} and \c{S}tefan~\cite{Ste99},
we classify Hopf superalgebras of dimension up to $5$ and determine their duals as follows.
\begin{theorem}[= Theorems~\ref{thm:ShiShiWak2}, \ref{thm:4ss} and \ref{thm:A_4-self-dual}]
Let $\HH$ be a Hopf superalgebra of dimension less than or equal to $5$.
If $\HH_{\bar1}\neq0$, then $\HH$ is isomorphic to one of the following Hopf superalgebras.
\begin{center}\footnotesize
\begin{tabular}{|c||l|l|l|}
\hline
\bf dim & \tabb{\bf Hopf superalgebras $\HH$\\ \bf with $\HH_{\bar1}\neq0$} & \tabb{\bf structures} & \tabb{\bf notes} \\
\hline
\hline
$2$ & \tabb{$\bigwedge(z)=\kk\langle z \mid z^2=0\rangle$\\ with $|z|=1$} & \tabb{$z$: odd primitive} & \tabb{pointed\\ self-dual} \\
\hline
$3$ & \tabb{None} & & \\
\hline
    & \tabb{$\HH_4^{(1)}=\bigwedge(z_1,z_2)$\\ $=\kk\langle z_1,z_2 \mid z_iz_j=-z_jz_i\rangle$\\ with $|z_1|=|z_2|=1$} & \tabb{$z_1,z_2$: primitive} & \tabb{pointed\\ self-dual} \\ \cline{2-4}
    & \tabb{$\HH_4^{(2)}$\\ $=\kk\langle g,z \mid g^2=1,z^2=0,gz=zg \rangle$\\ with $|g|=0,|z|=1$} & \tabb{$g$: group-like\\ $z$: odd primitive} & \tabb{pointed\\ self-dual} \\ \cline{2-4}
$4$ & \tabb{$\HH_4^{(3)}$\\ $=\kk\langle g,z \mid g^2=1,z^2=0,gz=zg \rangle$\\ with $|g|=0,|z|=1$} & \tabb{$g$: group-like\\ $z$: odd $g$-skew primitive} & \tabb{pointed\\ $(\HH_4^{(3)})^*$\\ $\cong \HH_4^{(4)}$} \\ \cline{2-4}
    & \tabb{$\HH_4^{(4)}$\\ $=\kk\langle g,z \mid g^2=1,z^2=0,gz=-zg \rangle$\\ with $|g|=0,|z|=1$} & \tabb{$g$: group-like\\ $z$: odd primitive} & \tabb{pointed\\ $(\HH_4^{(4)})^*$\\ $\cong \HH_4^{(3)}$} \\ \cline{2-4}
    & \tabb{$\A_4(\zeta_4), \A_4(-\zeta_4)$} &
\tabb{Example~\ref{ex:AM}} & \tabb{semisimple\\ self-dual} \\
\hline
$5$ & \tabb{None} & & \\
\hline
\end{tabular}
\end{center}
Note that $\HH_4^{(2)}\cong\kk\Z_2\otimes\bigwedge(z)$ and $\A_4(\zeta_4)\not\cong\A_4(-\zeta_4)$.
\end{theorem}

The method developed here has subsequently been applied to classify non-semisimple pointed Hopf superalgebras of dimension up to $10$ in~\cite{ShiShiWak24}.
In this paper, we also consider some classes of Hopf superalgebras in dimensions $6$ and $8$.
One example arises from the classification result of semisimple Hopf algebras of dimension $12$ by Fukuda~\cite{Fuk97}.
Using his result, we prove the following theorem.
\begin{theorem}[= Theorem~\ref{thm:H_6^5}] \label{thm:intro3}
A semisimple Hopf superalgebra $\HH$ of dimension $6$ with $\HH_{\bar{1}} \ne 0$ is isomorphic to the self-dual Hopf superalgebra $\A_6$ (see Proposition~\ref{prp:H_6^5} for the definition).
\end{theorem}

Garc\'ia and Vay~\cite{GarVay10} completed the classification of Hopf algebras of dimension $16$.
According to their result, a non-semisimple non-pointed self-dual Hopf algebra of this dimension is isomorphic to one of two Hopf algebras given in C\u{a}linescu, D\u{a}sc\u{a}lescu, Masuoka and Menini~\cite{CalDasMasMen04}.
We obtain the following theorem.
\begin{theorem}[= Theorem~\ref{thm:K_8}] \label{thm:intro4}
A non-semisimple Hopf superalgebra $\HH$ of dimension $8$ such that neither $\HH$ nor $\HH^*$ is pointed and $\HH_{\bar{1}} \ne 0$ is isomorphic to one of the following eight pairwise non-isomorphic Hopf superalgebras
\[
\mathcal{K}_8(\zeta;\epsilon,\eta) \qquad (\zeta\in\{\zeta_4,-\zeta_4\},\;\epsilon,\eta\in \{0,1\})
\]
given in Proposition~\ref{prp:K_8}.
Moreover, for each $\zeta\in\{\zeta_4,-\zeta_4\}$, we have
\[
\mathcal{K}_8(\zeta;0,0)^* \cong \mathcal{K}_8(\zeta;0,0),\;\;
\mathcal{K}_8(\zeta;1,1)^* \cong \mathcal{K}_8(\zeta;1,1),\;\;
\mathcal{K}_8(\zeta;0,1)^* \cong \mathcal{K}_8(\zeta;1,0)
\]
as Hopf superalgebras.
\end{theorem}

\subsection*{Organization of the paper}
The paper is organized as follows.
In Section~\ref{sec:boson},
we review the definition and properties of Yetter-Drinfeld modules (Section~\ref{subsec:YD}) and
the bosonization technique (Section~\ref{subsec:boson}) introduced by Majid \cite{Maj94} and Radford \cite{Rad85}.

In Section~\ref{sec:Hopf-super},
we review the definition and properties of Hopf superalgebras.
We first see that the category of superspaces is naturally embedded
in the category $\YD{\kk\Z_2}$ of left $\kk\Z_2$-Yetter-Drinfeld modules (Section~\ref{subsec:super-symmet}).
We next see some examples of Hopf superalgebras (Section~\ref{subsec:Hopf-super}).
We recall from \cite{AndEtiGel01} a construction of Hopf superalgebras (Section~\ref{subsub:AEG}).
For a finite-dimensional Hopf superalgebra,
as in the non-super situation, we see that its dual superspace becomes a finite-dimensional Hopf superalgebra again,
called the {\it dual Hopf superalgebra} (Section~\ref{subsec:duals}).

In Section~\ref{sec:bos},
we apply the Radford-Majid bosonization to Hopf superalgebras and study its properties (Section~\ref{subsec:boson-super}).
We say that a Hopf superalgebra $\HH$ is a {\it super-form} of a Hopf algebra $A$
if the bosonization $\widehat{\HH}$ of $\HH$ is isomorphic to $A$ as Hopf algebras (Definition~\ref{def:super-form}).
For a given Hopf algebra $A$,
in Section~\ref{subsec:YD-data}, we introduce the notion of {\it admissible data} for $A$ (Definition~\ref{def:yddata})
and show that there is a one-to-one correspondence between the equivalence classes of admissible data for $A$ and
the set of all isomorphism classes of Hopf algebras $\HH$ in $\YD{\kk\Z_2}$ such that the bosonization of $\HH$ is isomorphic to $A$ (Proposition~\ref{prp:one-to-one}).
In Section~\ref{subsec:coinv-sub},
we give a criterion for such $\HH$ to be a Hopf superalgebra (Propositions~\ref{prp:super-criteria1} and \ref{prp:super-criteria2}).
If an admissible datum for $A$ satisfies the conditions given in Propositions~\ref{prp:super-criteria1} and \ref{prp:super-criteria2},
we call it a {\it super-datum} for $A$ (Definition~\ref{def:SD}).

In the final Section~\ref{sec:class},
as a demonstration of our method,
we classify Hopf superalgebras of dimension up to $5$ (Sections~\ref{subsec:general}, \ref{subsec:pt-4-dim} and \ref{subsec:Mas}).
The proofs of Theorems~\ref{thm:intro3} and \ref{thm:intro4} are given in Section~\ref{subsec:Fuk} and in Section~\ref{subsec:CDMM}, respectively.

\subsection*{Acknowledgments}
We thank the organizers, Siu-Hung~Ng and Susan~Montgomery, of
``the AMS Special Session on Quantum Groups, Hopf algebras and Applications:
In honor of Professor Earl J.~Taft''
held at the University of Utah for giving us an opportunity to present the results.
We also thank Kenichi Shimizu for his helpful comments.
The first author is supported by the Japan Society for the Promotion of Science (JSPS) KAKENHI Grant Number JP22K13905.
We are grateful for the referee's careful reading and valuable comments.
Thanks to them, we were able to make some of the proofs more conceptual.

\section{The Radford-Majid bosonization} \label{sec:boson}
Throughout this paper, we work over a field $\kk$.
In Sections~\ref{sec:Hopf-super} and \ref{sec:bos}, we suppose $\kk$ to be of characteristic $\neq2$.
In Section~\ref{sec:class}, $\kk$ is supposed to be an algebraically closed field of characteristic zero.
The unadorned $\otimes$ is the tensor product $\otimes_{\kk}$ over $\kk$.

In this section, we review the bosonization (biproduct) technique introduced by Radford \cite{Rad85} and Majid \cite{Maj94}.
In the following, we fix a Hopf algebra $H=(H;m_H,u_H,\Delta_H,\varepsilon_H,S_H)$,
where $m_H$, $u_H$, $\Delta_H$, $\varepsilon_H$ and $S_H$
are the multiplication, the unit, the comultiplication, the counit and the antipode of $H$, respectively.
As usual, we denote the unit of $H$ by $1_H:=u_H(1)$ and sometimes write $1$ for $1_H$, omitting the subscript $H$.
We freely use the Heyneman-Sweedler notation such as $\Delta_H(h)=h_{(1)}\otimes h_{(2)}$ for $h\in H$.
Let $H^+$ denote the kernel of the counit $\varepsilon_H$ of $H$.

\subsection{Coinvariant subspaces} \label{subsec:coinv}
A right $H$-module $M$ is called a {\it right $H$-Hopf module} if
$M$ is simultaneously a right $H$-comodule satisfying
$(m.h)_{(0)}\otimes (m.h)_{(1)} = m_{(0)}.h_{(1)} \otimes m_{(1)} h_{(2)}$
for all $h\in H,m\in M$.
Here, $m.h$ denotes the right $H$-action of $h\in H$ on $m\in M$
and $M\to M\otimes H; m\mapsto m_{(0)}\otimes m_{(1)}$ denotes the right $H$-coaction on $M$.
Let $\rhmod{H}$ (resp.~$\Vect$) denote the category of all right $H$-Hopf modules (resp.~all vector spaces).
For $V\in\Vect$, we naturally regard $V\otimes H$ as a right $H$-Hopf module (by $\id_V\otimes m_H$ and $\id_V\otimes\Delta_H$),
and hence we get a functor $\Vect\to \rhmod{H}; V\mapsto V\otimes H$.

Let $M$ be a right $H$-Hopf module.
We denote the {\it $H$-coinvariant subspace} of $M$ by $M^{\co(H)} := \{ m\in M \mid m_{(0)}\otimes m_{(1)} = m\otimes 1_H \}$.
One sees that the map
\begin{equation} \label{eq:one}
M^{\co(H)} \otimes H \overset{\cong}\longrightarrow M;\quad
v\otimes h \longmapsto v.h
\end{equation}
is a natural isomorphism in $\rhmod{H}$ whose inverse is given by $m\mapsto E_M(m_{(0)})\otimes m_{(1)}$,
where
\begin{equation} \label{eq:zero}
E_M : M \longrightarrow M^{\co(H)};
\quad m \longmapsto m_{(0)}. S_H(m_{(1)}).
\end{equation}
This shows that the functor $\rhmod{H}\to \Vect; M\mapsto M^{\co(H)}$ is an equivalence whose quasi-inverse is
$\Vect\to\rhmod{H}; V\mapsto V\otimes H$ (the fundamental theorem of Hopf modules).

Let $M$ be a right $H$-Hopf module.
We set $\overline{M}^{H}:=M/MH^+$ and denote the canonical quotient map by $\overline{(-)}:M\to \overline{M}^H; m\mapsto \overline{m}$.
One sees that there is another equivalence $\rhmod{H}\approx \Vect$ given by $M\mapsto \overline{M}^{H}$.
The associated natural isomorphism in $\rhmod{H}$ is given by
\begin{equation} \label{eq:two}
M \overset{\cong}\longrightarrow \overline{M}^H\otimes H;\quad
m \longmapsto \overline{m_{(0)}}\otimes m_{(1)}.
\end{equation}
Indeed, the inverse is given as $\overline{m}\otimes h \mapsto E_M(m).h$.

The natural isomorphisms
\begin{equation} \label{eq:three}
M^{\co(H)} \overset{\cong}\longrightarrow \overline{M}^H; \;\; v\longmapsto \overline{v}
\quad\text{and}\quad
\overline{M}^H \overset{\cong}\longrightarrow M^{\co(H)}; \;\; \overline{m} \longmapsto E_M(m)
\end{equation}
are inverses of each other.
If we identify $M^{\co(H)}$ and $\overline{M}^H$ through these isomorphisms,
then \eqref{eq:one} and \eqref{eq:two} are seen to be inverses of each other.

\subsection{Yetter-Drinfeld categories} \label{subsec:YD}
In this section we suppose that the antipode $S_H$ of $H$ is bijective.
A left $H$-module $V$ is called a {\it left $H$-Yetter-Drinfeld module} if
$V$ is simultaneously a left $H$-comodule satisfying the following compatibility condition.
\[
(h.v)_{(-1)}\otimes (h.v)_{(0)} = h_{(1)}v_{(-1)}S_H(h_{(3)}) \otimes h_{(2)}.v_{(0)}
\quad (h\in H, v\in V).
\]
Here, $h.v$ denotes the left $H$-action of $h\in H$ on $v\in V$
and $V\to H\otimes V; v \mapsto v_{(-1)}\otimes v_{(0)}$ denotes the left $H$-coaction on $V$.
Let $\YD{H}$ denote the category of all left $H$-Yetter-Drinfeld modules.

The category $\YD{H}$ has a structure of a monoidal category inherited from the category of left $H$-modules and that of left $H$-comodules.
Moreover, one sees that $\YD{H}$ forms a braided monoidal category with braiding
\begin{equation} \label{eq:YD-braiding}
c_{V,W} :V\otimes W\longrightarrow W\otimes V;
\quad
v\otimes w\longmapsto v_{(-1)}.w \otimes v_{(0)} \quad (V,W\in\YD{H}).
\end{equation}

Therefore, we may consider {\it (co)algebras in $\YD{H}$} and {\it Hopf algebras in $\YD{H}$}.
Note that, in some literature, a Hopf algebra in $\YD{H}$ is simply called a {\it braided Hopf algebra} (if $H$ is obvious).
For later use,
we let $\YDH{H}$ denote the category consisting of all Hopf algebras in $\YD{H}$.

Let $\T_H$ denote the category defined as follows.
\begin{itemize}
\item (objects) An object is a triplet $(A;\iota,\pi)$ consisting of
a Hopf algebra $A$ and Hopf algebra maps $\iota:H\to A,\pi:A\to H$ satisfying $\pi\circ\iota=\id_H$.
\item (morphisms) A morphism $(A;\iota,\pi)\to (A';\iota',\pi')$ is a Hopf algebra map $\varphi:A\to A'$
which satisfies $\varphi\circ\iota=\iota'$ and $\pi'\circ \varphi=\pi$.
\end{itemize}
In \cite{CalDasMasMen04}, an object in $\T_H$ is called a {\it Hopf algebra triple} over $H$.

Let $(A;\iota,\pi)$ be an object in $\T_H$.
We denote the Hopf algebra structure of $A$ by $A=(A;m_A,u_A,\Delta_A,\varepsilon_A,S_A)$.
Then $m_A\circ(\id_A\otimes \iota)$ and $(\id_A\otimes\pi)\circ\Delta_A$ make $A$ into an object in $\rhmod{H}$.

Let $\B$ be the $H$-coinvariant subspace $A^{\co(H)}$ of $A$.
It is easy to see that $\B$ is a subalgebra of $A$ and $\B$ is a left $H$-comodule via $(\pi\otimes\id_A)\circ\Delta_A$,
and hence $\B$ becomes a left $H$-comodule algebra.
Since $AH^+$ is a coideal of $A$, the quotient $\overline{A}^H\;(=A/AH^+)$ is naturally a coalgebra
and $\overline{A}^H$ is a left $H$-module via $H\otimes \overline{A}^H\to \overline{A}^H; h\otimes\overline{a}\mapsto \overline{\iota(h)a}$.
In this way, $\overline{A}^H$ becomes a left $H$-module coalgebra.

Through the identification \eqref{eq:three},
we see that the left $H$-comodule algebra $\B\;(=\overline{A}^H)$ is simultaneously a left $H$-module coalgebra.
One sees that the induced left $H$-module structure on $\B$ coincides with the adjoint action
\[
H\otimes \B \longrightarrow \B;\quad
h\otimes b \longmapsto h\triangleright b := \iota(h_{(1)}) b \iota(S_H(h_{(2)}))
\]
and the induced coalgebra structure $\Delta_\B$ on $\B$ is
\[
\Delta_\B :\B\longrightarrow \B\otimes \B;\quad
b\longmapsto E_A(b_{(1)})\otimes b_{(2)} = b_{(1)} (\iota\circ S_H\circ\pi)(b_{(2)})\otimes b_{(3)},
\]
where $E_A$ is given by \eqref{eq:zero}.
Here, we write $\Delta_A(b)=b_{(1)}\otimes b_{(2)}$ for $b\in \B\;(\subset A)$.
A direct calculation shows that $\Delta_\B$ is an algebra map in $\YD{H}$.
Moreover, we have the following.
\begin{proposition}\label{prp:coinv-sub}
The left $H$-comodule algebra $\B=A^{\co(H)}$ becomes an object in $\YDH{H}$ by the following structures.
\begin{itemize}
\item (left $H$-action) $H\otimes \B \to \B;\; h\otimes b \mapsto h\triangleright b$.
\item (comultiplication) $\Delta_{\B}:\B\to \B\otimes \B;\; b\mapsto b_{(1)}(\iota\circ S_H\circ\pi)(b_{(2)})\otimes b_{(3)}$.
\item (counit) $\varepsilon_{\B}:\B\to \kk;\; b\mapsto \varepsilon_A(b)$.
\item (antipode) $S_{\B}:\B\to \B;\; b\mapsto (\iota\circ\pi)(b_{(1)}) S_A(b_{(2)})$.
\end{itemize}
\end{proposition}

In the following, we simply denote the Hopf algebra $\mathcal{B}$ in $\YD{H}$ of the above proposition by $A^{\co(\iota,\pi)}$ and call it the {\it coinvariant subalgebra} of $A$ with respect to $(\iota,\pi)$.

For a morphism $\varphi:(A;\iota,\pi)\to (A';\iota',\pi')$ in $\T_H$,
one sees that the restriction of $\varphi:A\to A'$ induces a morphism $A^{\co(\iota,\pi)}\to (A')^{\co(\iota',\pi')}$ in $\YDH{H}$.
Thus, we get a functor
\[
\mathcal{F} : \T_H \longrightarrow \YDH{H};
\quad
(A;\iota,\pi) \longmapsto A^{\co(\iota,\pi)}.
\]

\subsection{The Radford-Majid bosonization} \label{subsec:boson}
In this section, we also suppose that the antipode $S_H$ of $H$ is bijective.

Let $\B=(\B;m_\B,u_\B,\Delta_\B,\varepsilon_\B,S_\B)$ be a Hopf algebra in $\YD{H}$.
The following result is due to \cite{Rad85,Maj94}.

\begin{theorem}
The following structure makes $\B\biprod H:=\B\otimes H$ into a Hopf algebra.
For $b,b'\in \B$ and $h,h'\in H$,
\begin{itemize}
\item (multiplication) $(b\biprod h)(b'\biprod h'):=b(h_{(1)}.b')\biprod h_{(2)}h'$,
\item (unit) $1_{\B\biprod H}:=1_\B\biprod 1_H$,
\item (comultiplication) $\Delta_{\B\biprod H}(b\biprod h):=(b_{(1)}\biprod (b_{(2)})_{(-1)}h_{(1)})\otimes((b_{(2)})_{(0)}\biprod h_{(2)})$,
\item (counit) $\varepsilon_{\B\biprod H}(b\biprod h):=\varepsilon_\B(b)\varepsilon_H(h)$,
\item (antipode) $S_{\B\biprod H}(b\biprod h):=(1_{\B}\biprod  S_H(b_{(-1)}h))(S_\B(b_{(0)})\biprod 1_{H})$.
\end{itemize}
Here, we denote $b\otimes h$ by $b\biprod h$ for $b\in \B$ and $h\in H$.
Moreover, if we set
\[
\pi_{\B\biprod H}(b\biprod h) = \varepsilon_\B(b)h
\quad\text{and}\quad
\iota_{\B\biprod H}(h) = 1_\B\biprod  h
\quad (b\in\B,h\in H),
\]
then $\B \mapsto (\B\biprod H;\iota_{\B\biprod H},\pi_{\B\biprod H})$ gives rise to a functor $\mathcal{G}: \YDH{H} \to \T_H$.
\end{theorem}

The Hopf algebra $\B\biprod H$ constructed above is called the {\it bosonization} (or {\it biproduct}) of $\B$ by $H$.
Since $(\B\biprod H)^{\co(H)}\cong \B$, one easily sees that $\mathcal{F}\circ\mathcal{G}\simeq \id$.

Let $(A;\iota,\pi)$ be an object in $\T_H$.
As we have seen in Section~\ref{subsec:YD}, $\B:=A^{\co(\iota,\pi)}$ is a left $H$-module algebra,
and hence the isomorphism given in \eqref{eq:one} is now an isomorphism
\[
\B \ratimes H \overset{\cong}\longrightarrow A;\quad b\otimes h \longmapsto b\iota(h)
\]
of right $H$-comodule algebras,
where $\B \ratimes H$ is the smash product of $\B$ and $H$.
Similarly, the isomorphism given in \eqref{eq:two} is now an isomorphism
\[
A \overset{\cong}\longrightarrow \overline{A}^H \rctimes H;\quad a \longmapsto \overline{a_{(1)}} \otimes \pi(a_{(2)})
\]
of right $H$-module coalgebras,
where $\overline{A}^H \rctimes H$ is the smash coproduct of $\overline{A}^H$ and $H$.
As a result, we get a Hopf algebra isomorphism
\begin{equation} \label{eq:four}
\B \biprod  H \overset{\cong}\longrightarrow A;\quad b\otimes h \longmapsto b\iota(h)
\end{equation}
through the identification $\B=\overline{A}^H$ given in \eqref{eq:three}.
This proves $\mathcal{G}\circ\mathcal{F}\simeq \id$.

We have thus proved the following theorem,
which is a categorical reformulation of Radford's results~\cite[Theorems~1--3]{Rad85}.
In fact, the reformulation is due to C\u{a}linescu, D\u{a}sc\u{a}lescu, Masuoka, and Menini~\cite[Proposition~1.1]{CalDasMasMen04}
who, however, did not give a proof.
\begin{theorem}
$\mathcal{G}:\YDH{H}\to\T_H$ is an equivalence with quasi-inverse $\mathcal{F}$.
\end{theorem}

For $(A;\iota,\pi),(A;\iota',\pi')\in \T_H$,
we write $(\iota,\pi)\approx(\iota',\pi')$ if there exists a Hopf algebra automorphism $\varphi:A\to A$
such that $\iota'=\varphi\circ\iota$ and $\pi=\pi'\circ\varphi$.
Then $\approx$ is an equivalence relation and we immediately get the following result.
\begin{proposition} \label{prp:YD-one-to-one}
Let $A$ be a fixed Hopf algebra.
There exists a one-to-one correspondence between the following two sets.
\begin{itemize}
\item The isomorphism classes of Hopf algebras $\B$ in $\YD{H}$
such that $\B\biprod H$ is isomorphic to $A$ as a Hopf algebra.
\item The equivalence classes of pairs $(\iota,\pi)$ such that $(A;\iota,\pi) \in \T_H$.
\end{itemize}
\end{proposition}

\subsection{Duals of coinvariant subalgebras} \label{subsec:dual-bos}
In this section, we suppose that $H$ is finite-dimensional.
Let $\T_H^\fd$ (resp.~$\YDH{H}^\fd$) denote the full subcategory of $\T_H$ (resp.~$\YDH{H}$) consisting of all finite-dimensional objects.
The assignment $(A;\iota,\pi) \mapsto (A^*;\pi^*,\iota^*)$ gives a category anti-equivalence between $\T_H^\fd$ and $\T_{H^*}^\fd$.

In the following, we fix $\B\in \YDH{H}^\fd$ and put $(A;\iota,\pi):=\mathcal{G}(\B)\;\in \T_H^\fd$.
Note that $\B=A^{\co(H)}=\mathcal{F}(A;\iota,\pi)$.
Since $A^*$ is naturally a right $H^*$-module (resp.~right $H^*$-comodule) by
$\Delta_A^*\circ(\id\otimes \pi^*)$ (resp.~$(\id\otimes \iota^*)\circ m_A^*$),
we may consider $A^{*\co(H^*)}$ and $\overline{A^*}^{H^*}$.
See Section~\ref{subsec:coinv} for the notation.
One easily sees that $(\overline{A}^H)^* \cong A^{*\co(H^*)}$ and $\B^*\cong \overline{A^*}^{H^*}$.

Recall that the left $H$-module coalgebra structure on $\overline{A}^H$ and the left $H$-comodule algebra structure on $\B$
are given as
\[
H\otimes\overline{A}^H\longrightarrow \overline{A}^H;\;\; h\otimes \overline{a}\longmapsto \overline{\iota(h)a}
\quad\text{and}\quad
\B\longrightarrow H\otimes \B;\;\; b\longmapsto \pi(b_{(1)})\otimes b_{(2)},
\]
respectively.
By taking the linear dual, we see that $\overline{A^*}^{H^*}$ and $(A^*)^{\co(H^*)}$ become a left $H^*$-comodule algebra and a left $H^*$-module coalgebra,
respectively.

By the above structures and the identification \eqref{eq:three},
we see that
$A^{*\co(H^*)}\;(=\overline{A^*}^{H^*})$ coincides with $\mathcal{F}(A^*;\pi^*,\iota^*)=A^{*\co(\pi^*,\iota^*)}$ in $\YDH{H^*}$.
Therefore,
under the identification \eqref{eq:three},
the duals of the left $H$-Yetter-Drinfeld structure maps $H\otimes \B\to \B$ and $\B\to H\otimes \B$ of $\B$
make $\B^* \cong (\overline{A}^{H})^* \cong A^{*\co(H^*)}$ into a Hopf algebra in $\YD{H^*}$,
and it is naturally isomorphic to $A^{*\co(\pi^*,\iota^*)}$.

\begin{center}
\begin{tikzcd}
\T_{H}^\fd \ar[r, "\approx"] \arrow[d, "\approx"'] & \T_{H^*}^\fd \ar[d, "\approx"] \\
\YDH{H}^\fd \ar[r, "\approx"] & \YDH{H^*}^\fd,
\end{tikzcd}
\hspace{0.7cm}
\begin{tikzcd}
(A;\iota,\pi) \ar[r, mapsto] \arrow[d, mapsto] & (A^*;\pi^*,\iota^*) \ar[d, mapsto] \\
\B=A^{\co(\iota,\pi)} \ar[r, mapsto] & \B^* \cong A^{*\co(\pi^*,\iota^*)}.
\end{tikzcd}
\end{center}

For $\B\in\YDH{H}^\fd$ with $(A;\iota,\pi):=\mathcal{G}(\B)$,
we get the following commutative diagram.
\begin{center}
\begin{tikzcd}
(\B\biprod H) \times (\B^*\biprod H^*) \ar[dr] \arrow[d, "\cong"'] \\
A\times A^* \ar[r, "\langle\text{ , }\rangle"'] & \kk,
\end{tikzcd}
\hspace{0.7cm}
\begin{tikzcd}
(b\biprod h,\; f \biprod g) \ar[dr, mapsto] \arrow[d, mapsto] \\
(b \iota(h),\; f *\pi^*(g)) \ar[r, mapsto] & f(b)g(h).
\end{tikzcd}
\end{center}
Here, $\langle\;,\;\rangle$ is the natural evaluation (the canonical pairing) of $A$
and $f*\pi^*(g)$ is the convolution product of $f$ and $\pi^*(g)\;(=g\circ\pi)$.

\section{Hopf superalgebras} \label{sec:Hopf-super}
In this section, we suppose that the characteristic of $\kk$ is not $2$.
We refer to \cite{AndAngYam11} for further background on Hopf superalgebras.
Let $\Z_2=\Z/2\Z=\{\bar0,\bar1\}$ denote the additive group of order two.
We sometimes identify $\Z_2$ with the multiplicative group $\{\eb,\sigmab\}$ of order two,
where $\eb$ is the identity element and $\sigmab^2=\eb$.

\subsection{Superspaces} \label{subsec:super-symmet}
Let $\kk\Z_2$ be the group algebra of $\Z_2$ over $\kk$.
An object of the category $\lcomod{\kk\Z_2}$ of left $\kk\Z_2$-comodules is just a $\Z_2$-graded vector space $V=V_{\bar0}\oplus V_{\bar1}$.
The $\bar0$-component $V_{\bar0}$ (resp.~$\bar1$-component $V_{\bar1}$) of $V$ is called an {\it even part} (resp.~{\it odd part}) of $V$.
For a homogeneous element $0\neq v\in V_{\bar0}\cup V_{\bar1}$, we denote its degree by $|v|\;(\in\{0,1\})$.
For simplicity, when we write $|v|$, $v$ is always supposed to be homogeneous.

Since $\kk\Z_2$ is a Hopf algebra, $\lcomod{\kk\Z_2}$ naturally becomes a monoidal category.
Note that our base field $\kk$ is supposed to be a $\Z_2$-graded vector space with $\kk_{\bar0}=\kk$ and $\kk_{\bar1}=0$.
Moreover, the following {\it supersymmetry} makes $\lcomod{\kk\Z_2}$ into a symmetric monoidal category.
\begin{equation} \label{eq:super-symmet}
\stw_{V,W} : V\otimes W \longrightarrow W\otimes V;
\quad
v\otimes w \longmapsto (-1)^{|v||w|} w\otimes v
\quad (V,W\in \lcomod{\kk\Z_2}).
\end{equation}
We denote this symmetric monoidal category $\lcomod{\kk\Z_2}$ by $\sVect$.
Its full subcategory of finite-dimensional objects is a symmetric tensor category.
An object of $\sVect$ is called a {\it superspace}.
A superspace $V=V_{\bar0}\oplus V_{\bar1}$ is said to be {\it purely even} if the odd part $V_{\bar1}$ of $V$ is zero.
Obviously, $\kk^{1|0}:=\kk\oplus0$ and $\kk^{0|1}:=0\oplus\kk$ are simple objects in $\sVect$.
Conversely, one sees that all simple objects are exhausted by $\kk^{1|0}$ and $\kk^{0|1}$.

In the following, we let $\eb^*,\sigmab^*\in(\kk\Z_2)^*$ denote the elements dual to $\eb,\sigmab\in\kk\Z_2$, respectively.
The following is obvious.
\begin{lemma} \label{prp:kZ_2}
The linear map $\kk\Z_2\to (\kk\Z_2)^*; \sigmab^i\mapsto \eb^*+(-1)^i\sigmab^*$ is a Hopf algebra isomorphism.
In particular, $\kk\Z_2$ is self-dual.
\end{lemma}

Let $V=V_{\bar0}\oplus V_{\bar1}$ be an object in $\sVect$.
By Lemma~\ref{prp:kZ_2}, $V$ naturally becomes a left $\kk\Z_2$-module with structure
\[
\kk\Z_2\otimes V \longrightarrow V; \quad
\sigmab^i \otimes v \longmapsto \sigmab^i.v := v_{\bar0} + (-1)^{i} v_{\bar1},
\]
where $v=v_{\bar0}+v_{\bar1}$ with $v_{\bar0}\in V_{\bar0}$ and $v_{\bar1}\in V_{\bar1}$.
If $v\in V$ is homogeneous, then we can simply write $\sigmab^i.v=(-1)^{i|v|}v$ for $i\in\{0,1\}$.

The following result tells us that a Hopf algebraic approach
can be employed in the study of the theory of super-mathematics (see Masuoka~\cite[Section~2]{Mas05} for example).
\begin{lemma} \label{prp:braided-mono-subcat}
The category $\sVect$ is a monoidal full subcategory of $\YD{\kk\Z_2}$, and the braiding of $\YD{\kk\Z_2}$ restricts to the supersymmetry.
\end{lemma}
\begin{proof}
Let $V=V_{\bar0}\oplus V_{\bar1}\in \sVect$.
First, note that the left $\kk\Z_2$-comodule structure $\delta:V\to \kk\Z_2\otimes V$ is explicitly given by
$\delta(v) = \sigmab^i \otimes v$ for $v\in V_{\bar i}$.
For $i\in\{0,1\}$ and $v=v_{\bar0}+v_{\bar1}$ with $v_{\bar0}\in V_{\bar0}$ and $v_{\bar1}\in V_{\bar1}$, we have
\[
\delta(\sigmab^i. v) = \delta(v_{\bar0}) + (-1)^i \delta(v_{\bar1}) = \eb\otimes v_{\bar0} + (-1)^i\sigmab\otimes v_{\bar1}.
\]
This shows that $V$ is an object of $\YD{\kk\Z_2}$.
Take $V,W\in\sVect$ and $v\in V,w\in W$ to be homogeneous.
Since $\delta(v)=\sigmab^{|v|}\otimes v$ and $\sigmab^{|v|}.w = (-1)^{|v||w|}w$,
the braiding given in \eqref{eq:YD-braiding} is calculated as
\[
c_{V,W}(v\otimes w) = \sigmab^{|v|}.w\otimes v = (-1)^{|v||w|} w \otimes v = \stw_{V,W}(v\otimes w).
\]
This completes the proof.
\end{proof}

A {\it superalgebra} (resp.~{\it supercoalgebra}) is an algebra (resp.~coalgebra) in $\sVect$ ($\subset \YD{\kk\Z_2}$).
Let $\A=(\A;m_\A,u_\A)$ be a superalgebra.
We say that $\A$ is {\it super-commutative} if it satisfies $ab=(-1)^{|a||b|}ba$ for all $a,b\in\A$.
A superspace $V\in\sVect$ is called a {\it left $\A$-supermodule}
if it is equipped with a morphism $\rho:\A\otimes V\to V$ in $\sVect$ such that $\rho\circ(m_\A\otimes \id_V) = \rho\circ(\id_\A\otimes\rho)$ and $\rho\circ(u_\A\otimes\id_V)=\id_V$.
The category of all left $\A$-supermodules is denoted by $\lsmod{\A}$.
We say that a finite-dimensional superalgebra $\A$ is {\it semisimple} if the category $\lsmod{\A}$ is semisimple, that is, every object of $\lsmod{\A}$ is a direct sum of simple objects.

Let $\C=(\C;\Delta_\C,\varepsilon_\C)$ be a supercoalgebra.
We say that $\C$ is {\it super-cocommutative} if it satisfies $\Delta_\C(c) = (-1)^{|c_{(1)}||c_{(2)}|}c_{(2)}\otimes c_{(1)}$ for all $c\in \C$.
The notion of right $\C$-supercomodule is defined in a similar manner.
The category of all right $\C$-supercomodules is denoted by $\rscomod{\C}$.
We say that $\C$ is {\it pointed} if any simple right $\C$-supercomodule is one-dimensional.

\subsection{Hopf superalgebras} \label{subsec:Hopf-super}
A {\it Hopf superalgebra} is a Hopf algebra in the symmetric monoidal category $\sVect$.
One sees that the antipode $S_\HH$ of a Hopf superalgebra $\HH$ is a {\it super-anti-algebra map}, that is, it satisfies
\begin{equation}\label{eq:antipode}
S_\HH(hh') = (-1)^{|h||h'|} S_\HH(h') S_\HH(h)
\end{equation}
for $h,h'\in\HH$.
Also, one can show that $S_\HH$ is a {\it super-anti-coalgebra map} (we omit the definition here since we do not use it in this paper).

Let $\HH=(\HH;m_{\HH},u_{\HH},\Delta_{\HH},\varepsilon_{\HH},S_{\HH})$ be a Hopf superalgebra.
Note that the comultiplication $\Delta_{\HH}$ of $\HH$ satisfies
$\Delta_\HH(\HH_{\bar\epsilon})\subset \sum_{\bar\eta+\bar\eta'=\bar\epsilon}\HH_{\bar\eta}\otimes \HH_{\bar\eta'}$ for $\epsilon\in\{0,1\}$.
As in the non-super setting, the set
\[
\G(\HH) := \{g\in \HH_{\bar0} \mid \varepsilon_\HH(g)=1,\,\,\Delta_{\HH}(g)=g\otimes g\}
\]
naturally becomes a group with $g^{-1}=S_{\HH}(g)$ for $g\in \G(\HH)$.
An element of $\G(\HH)$ is called a {\it group-like} element of $\HH$ (cf.~Remark~\ref{rem:A_4}).

For a fixed $g\in \G(\HH)$,
an element $z\in \HH$ is said to be {\it $g$-skew primitive} if it satisfies
\[
\Delta_\HH(z) = g\otimes z + z \otimes 1_\HH.
\]
In this case, we see that $\varepsilon_{\HH}(z)=0$ and $S_{\HH}(z)=-g^{-1}z$.
If a $g$-skew primitive element $z$ belongs to $\HH_{\bar1}$, then we shall call $z$ an {\it odd $g$-skew primitive}.
A $1_\HH$-skew primitive element is called a {\it primitive} element, and an odd $1_\HH$-skew primitive element is called an {\it odd primitive} element.
The set of all primitive elements in $\HH$ is denoted by $\PP(\HH)$.
One easily sees that $\PP(\HH)$ becomes a {\it Lie superalgebra} (i.e., a Lie algebra in $\sVect$) with {\it superbracket} $[x,y]:=xy-(-1)^{|x||y|}yx$ ($x,y\in\PP(\HH)$).

\begin{example}
Let $V$ be a finite-dimensional vector space with basis $\{z_1,\dots,z_\theta\}$.
The exterior algebra
\[
\HH:= \bigwedge(V)=\bigwedge(z_1,\dots,z_\theta)
\]
over $V$ naturally becomes a super-commutative superalgebra, called the {\it exterior superalgebra}.
Moreover, $\HH$ becomes a super-commutative and super-cocommutative Hopf superalgebra by declaring $z_1,\dots,z_\theta$ to be odd primitive.
One sees that $\G(\HH)=\{1\}$, $\PP(\HH)=V$ and $\HH$ is pointed.
\qedheree
\end{example}

We assume our base field $\kk$ is an algebraically closed field of characteristic zero throughout this Section~\ref{subsec:Hopf-super}.

Let $\HH$ be a super-cocommutative Hopf superalgebra.
Then the group $\Gamma:=\G(\HH)$ acts on $\mathfrak{g}:=\PP(\HH)$ by the adjoint action as usual
and the universal enveloping superalgebra $\mathcal{U}(\mathfrak{g})$ of the Lie superalgebra $\mathfrak{g}$
forms a left $\kk\Gamma$-supermodule.
The following result was first shown by Kostant~\cite[Theorem~3.3]{Kos77}
(see also Andruskiewitsch, Etingof and Gelaki~\cite[Corollary~2.3.5]{AndEtiGel01} and Masuoka~\cite[Theorems~3.6 and 4.5]{Mas05}).
\begin{theorem}
$\HH$ is isomorphic to the smash product $\kk\Gamma\ltimes\mathcal{U}(\mathfrak{g})$ of $\kk\Gamma$ and $\mathcal{U}(\mathfrak{g})$.
In particular, if $\HH$ is finite-dimensional, then $\HH\cong \kk\Gamma\ltimes\bigwedge(\mathfrak{g})$.
\end{theorem}

The finite-dimensional super-commutative and super-cocommutative cases are therefore already understood.

\begin{examples} \label{ex:AM}
Fix a primitive fourth root of unity $\zeta_4\in\kk$.
The following complete list of pairwise non-isomorphic Hopf superalgebras of dimension $4$ with non-zero odd part is based on the classification of Aissaoui and Makhlouf~\cite{AisMak14}.
\begin{itemize}
\item $\HH_4^{(1)}:=\bigwedge(z_1,z_2)$.
\item $\HH_4^{(2)}:=\kk\langle g,z \mid g^2=1,z^2=0,gz=zg \rangle$,
where $g$ is group-like and $z$ is odd primitive.
Note that $\HH_4^{(2)}\cong\kk \Z_2\otimes \bigwedge(z)$.
\item $\HH_4^{(3)}:=\kk\langle g,z \mid g^2=1,z^2=0,gz=zg \rangle$,
where $g$ is group-like and $z$ is odd $g$-skew primitive.
\item $\HH_4^{(4)}:=\kk\langle g,z \mid g^2=1,z^2=0,gz=-zg \rangle$,
where $g$ is group-like and $z$ is odd primitive.
\item $\A_4(\pm\zeta_4):=\kk\langle x,z \mid x^2 + z^2=1,xz=zx=0 \rangle$,
where $x$ is even, $z$ is odd and
\[
\begin{gathered}
\Delta(x)=x\otimes x \pm \zeta_4 z\otimes z,\quad \varepsilon(x)=1,\quad S(x)=x, \\
\Delta(z)=x\otimes z + z\otimes x,\quad \varepsilon(z)=0,\quad S(z)=\mp\zeta_4 z.
\end{gathered}
\]
\end{itemize}
The antipodes of $\A_4(\zeta_4)$ and $\A_4(-\zeta_4)$ satisfy the super-anti-algebra identity~\eqref{eq:antipode}.
The two choices are non-isomorphic, since their antipodes have different eigenvalues.
As we have mentioned, the exterior superalgebra $\HH_4^{(1)}=\bigwedge(z_1, z_2)$ is pointed.
We will see that $\HH_4^{(2)}$, $\HH_4^{(3)}$ and $\HH_4^{(4)}$ are pointed, and that $\A_4(\zeta_4)$ and $\A_4(-\zeta_4)$ are semisimple.
See Theorem~\ref{thm:ShiShiWak2} and Example~\ref{ex:A_4}, respectively.
\qedheree
\end{examples}

\subsection{A construction of Hopf superalgebras} \label{subsub:AEG}
We recall from \cite{AndEtiGel01} a construction of Hopf superalgebras.
Let $H=(H;m_H,u_H,\Delta_H,\varepsilon_H,S_H)$ be a Hopf algebra, and let $c\in\G(H)$ be a group-like element such that $c^2=1_H$.
From the pair $(H, c)$, we can construct a Hopf superalgebra $\HH$ as follows:
As an algebra, $\HH = H$.
We make it into a $\Z_2$-graded vector space by letting
\[
\HH_{\bar\epsilon}:=\{ h\in \HH \mid chc=(-1)^\epsilon h\} \qquad (\epsilon \in\{0,1\}).
\]
For each $h\in \HH$, we define $\Delta_{\HH,\bar0}(h)\in \HH\otimes \HH_{\bar0}$ and $\Delta_{\HH,\bar1}(h)\in \HH\otimes \HH_{\bar1}$
satisfying $\Delta_H(h)=\Delta_{\HH,\bar0}(h)+\Delta_{\HH,\bar1}(h)$.

\begin{theorem}[\text{\cite[Theorem~3.1.1]{AndEtiGel01}}] \label{thm:AEG}
The superalgebra $\HH$ becomes a Hopf superalgebra together with the comultiplication, the counit and the antipode given by
\begin{equation} \label{eq:AEG-Delta-S}
\Delta_{\HH}(h) := \Delta_{\HH,\bar0}(h) - (-1)^{|h|}(c\otimes1_H)\Delta_{\HH,\bar1}(h),
\end{equation}
$\varepsilon_{\HH}:=\varepsilon_{H}$ and $S_{\HH}(h) := c^{|h|} S_H(h)$ for $h\in \HH$, respectively.
This construction establishes a bijective correspondence between:
\begin{itemize}
\item Pairs $(H,c)$ consisting of a Hopf algebra $H$ and an element $c\in\G(H)$ satisfying $c^2=1_H$.
\item Pairs $(\HH,g)$ consisting of a Hopf superalgebra $\HH$ and an element $g\in\G(\HH)$
satisfying $g^2=1_{\HH}$ and $gzg=(-1)^{|z|}z$ for all $z\in\HH$.
\end{itemize}
\end{theorem}

\begin{remark} \label{rem:A_4}
For this theorem to hold, a group-like element of $\HH$ needs to be homogeneous, and this is one of the reasons why we adopt our definition of $\G(\HH)$.
There may exist a non-homogeneous element $g\in \HH$ such that $\varepsilon(g)=1$ and $\Delta(g)=g\otimes g$.
For example, we shall consider the Hopf superalgebra $\HH=\A_4(-\zeta_4)$ given in Example~\ref{ex:AM}.
By definition, we have $\G(\HH)=\{1_{\HH},x^2-z^2\}$.
Let $\zeta_8$ be an element in $\kk$ satisfying $\zeta_8^2=\zeta_4$.
Then $g := x + \zeta_8^3 z$ satisfies $\varepsilon(g)=1$ and $\Delta(g)=g\otimes g$.
However, $g$ is non-homogeneous, and hence $g\notin\HH_{\bar0}$.
\qedheree
\end{remark}

We will see that a Hopf superalgebra obtained from the group algebra $\kk \sym$ by Theorem~\ref{thm:AEG} is, up to isomorphism, the unique semisimple Hopf superalgebra of dimension $6$ with non-zero odd part (see Section~\ref{subsec:Fuk}).
One of the four-dimensional Hopf superalgebras mentioned in Example~\ref{ex:AM} is also obtained by the construction of Theorem~\ref{thm:AEG}.

\begin{example}
Let $H_4$ denote {\it Sweedler's four-dimensional Hopf algebra}
\begin{equation} \label{eq:Sweedler}
H_4 = \kk\langle c,x \mid c^2=1,x^2=0,cx=-xc \rangle,
\end{equation}
where $c$ is group-like and $x$ is $c$-skew primitive.
One sees that the corresponding Hopf superalgebra of $(H_4,c)$ is $(\HH_4^{(4)},g)$,
where $\HH_4^{(4)}$ is given in Example~\ref{ex:AM}.
\qedheree
\end{example}

\subsection{Duals of Hopf superalgebras} \label{subsec:duals}
For superspaces $V,V'\;(\in\sVect)$, a bilinear map $\langle\;,\;\rangle:V'\times V\to \kk$ is said to be a {\it pairing on $V'$ and $V$} if
it satisfies $\langle V'_{\bar\epsilon}, V_{\bar\eta}\rangle=0$ if $\epsilon\neq \eta$ ($\epsilon,\eta\in\{0,1\}$).
In other words, the induced map $\langle\;,\;\rangle:V'\otimes V\to \kk^{1|0}$ is a morphism in $\sVect$.
Let $\langle\;,\;\rangle':W'\times W\to \kk$ be a pairing on $W'$ and $W$.
Then we define a pairing
\[
(V'\otimes W')\times (V\otimes W) \longrightarrow \kk;
\quad
(f\otimes g, v\otimes w) \longmapsto \langle f,v \rangle \langle g,w\rangle'
\]
on $V'\otimes W'$ and $V\otimes W$.
If $V'=W', V=W$ and $\langle\;,\;\rangle'=\langle\;,\;\rangle$,
then we simply denote the pairing on $V'\otimes V'$ and $V\otimes V$ by the same symbol
\[
\langle f\otimes g, v\otimes w\rangle = \langle f,v \rangle \langle g,w\rangle
\quad(f,g\in V', v,w\in V)
\]
as usual.

\begin{definition}
For Hopf superalgebras $\KK=(\KK;m_\KK,u_\KK,\Delta_\KK,\varepsilon_\KK,S_\KK)$ and $\HH=(\HH;m_\HH,u_\HH,\Delta_\HH,\varepsilon_\HH,S_\HH)$,
a pairing $\langle\;,\;\rangle:\KK\times\HH\to\kk$ on $\KK$ and $\HH$ is called a {\it Hopf pairing} if it satisfies the following conditions.
\[
\begin{gathered}
\langle k, hh' \rangle = \langle \Delta_\KK(k), h\otimes h' \rangle \;(= \langle k_{(1)},h\rangle \langle k_{(2)},h'\rangle), \quad
\langle k, 1_\HH \rangle = \varepsilon_\KK(k),\\
\langle kk',h \rangle = \langle k\otimes k', \Delta_\HH(h) \rangle \;(=\langle k,h_{(1)}\rangle \langle k',h_{(2)}\rangle), \quad
\langle 1_\KK, h\rangle = \varepsilon_\HH(h)
\end{gathered}
\]
for $k,k'\in\KK$ and $h,h'\in\HH$.
\qedheree
\end{definition}
In this case,
we have $\langle k, S_{\HH}(h) \rangle = \langle S_\KK(k), h\rangle$ for $k\in\KK, h\in\HH$,
as in the non-super setting.

Let $V\in\sVect$, and let $V^*$ denote the linear dual space of $V$ (over $\kk$).
By letting $(V^*)_{\bar\epsilon} := (V_{\bar\epsilon})^*$ ($\epsilon\in\{0,1\}$),
we make $V^*$ into an object of $\sVect$.
Then the evaluation map
\[
V^*\times V\longrightarrow\kk;
\quad (f,v) \longmapsto f(v)
\]
is a pairing on $V^*$ and $V$.

For a finite-dimensional Hopf superalgebra $\HH$,
one can make $\HH^*$ into a Hopf superalgebra, called the {\it dual Hopf superalgebra} of $\HH$,
so that the evaluation map $\langle \;,\;\rangle :\HH^*\times \HH \to \kk$ is a (non-degenerate) Hopf pairing.
Since $\HH\in\YDH{\kk\Z_2}^\fd$ and $(\kk\Z_2)^*\cong \kk\Z_2$,
the linear dual $\HH^*$ of $\HH$ may be regarded as an object in $\YDH{\kk\Z_2}^\fd$ (see Section~\ref{subsec:dual-bos}).
One easily sees that $\HH^*$ actually is an object in $\sVect$ and its Hopf superalgebra structure coincides with the one defined above.

\begin{remark}
Some literature uses another definition of the dual Hopf superalgebras which we shall explain below.
For pairings $\langle\;,\;\rangle:V'\times V\to \kk$ and $\langle\;,\;\rangle':W'\times W\to \kk$,
we note that there is another way to define a pairing on $V'\otimes W'$ and $V\otimes W$ as follows.
\[
(V'\otimes W')\times (V\otimes W) \longrightarrow \kk;
\quad
(f\otimes g, v\otimes w) \longmapsto (-1)^{|g||v|}\langle f,v \rangle \langle g,w\rangle'.
\]
Let $\HH$ be a finite-dimensional Hopf superalgebra.
If we use this pairing,
then we can also make the linear dual $\HH^*$ of $\HH$ into a Hopf superalgebra, which we denote by $\HH^{\bar*}$ (just here), satisfying
\[
\begin{gathered}
\langle f, hh' \rangle = (-1)^{|f_{(2)}||h|}\langle f_{(1)},h\rangle \langle f_{(2)},h'\rangle, \quad
\langle f, 1_\HH \rangle = f(1_\HH),\\
\langle fg,h \rangle = (-1)^{|h_{(1)}||g|}\langle f,h_{(1)}\rangle \langle g,h_{(2)}\rangle, \quad
\langle 1_{\HH^{\bar*}}, h\rangle = \varepsilon_\HH(h)
\end{gathered}
\]
for $f,g\in\HH^*$ and $h,h'\in\HH$.
Suppose that our base field $\kk$ contains a primitive fourth root of unity $\zeta_4$.
Then one sees that the map $\HH^*\to \HH^{\bar*}; f\mapsto \zeta_4^{|f|}f$ is a Hopf superalgebra isomorphism.
See \cite[Sections~3.1 and 3.2]{MasShi18} for details (see also \cite[Section~1]{Mas08}).
\qedheree
\end{remark}

\begin{example}
Let $V$ be a finite-dimensional vector space with basis $\{z_1,\dots,z_\theta\}$.
Then the evaluation map $V^*\times V\to\kk$ extends to a non-degenerate Hopf pairing $\langle\;,\;\rangle:\bigwedge(V^*)\times\bigwedge(V)\to\kk$ defined by
\[
\langle f_1\wedge \cdots \wedge f_n, v_1\wedge \cdots \wedge v_m \rangle = \delta_{n,m}\; \det\big( f_j(v_i) \big)_{i,j}
\qquad (n,m\in\mathbb{N}),
\]
where $\delta_{n,m}$ is the Kronecker symbol.
In particular, $\bigwedge(z_1,\dots,z_\theta)$ is a self-dual Hopf superalgebra.
\qedheree
\end{example}

\begin{example} \label{ex:AM-duals}
Recall the Hopf superalgebras $\HH_4^{(2)},\HH_4^{(3)},\HH_4^{(4)}$ given in Example~\ref{ex:AM}.
One sees that the pairings
\[
\langle\;,\;\rangle : \HH_4^{(2)}\times \HH_4^{(2)} \to \kk;\quad
\langle g,g \rangle =-1,\;
\langle z,z \rangle =1,\;
\langle g,z \rangle =\langle z,g \rangle =0
\]
and
\[
\langle\;,\;\rangle : \HH_4^{(3)}\times \HH_4^{(4)} \to \kk;\quad
\langle g,g \rangle =-1,\;
\langle z,z \rangle =1,\;
\langle g,z \rangle =\langle z,g \rangle =0
\]
are non-degenerate Hopf pairings.
In particular, $\HH_4^{(2)}$ is self-dual and the dual of $\HH_4^{(3)}$ is isomorphic to $\HH_4^{(4)}$.
\qedheree
\end{example}

\section{Bosonizations and super-forms} \label{sec:bos}
In this section, we also suppose that the characteristic of $\kk$ is not $2$.

\subsection{Bosonization of Hopf superalgebras} \label{subsec:boson-super}
In the following, we fix a Hopf superalgebra $\HH=(\HH;m_{\HH},u_{\HH},\Delta_{\HH},\varepsilon_{\HH},S_{\HH})$.
Since $\HH$ ($\in\sVect$) can be regarded as an object in
the category of left $\kk\Z_2$-Yetter-Drinfeld modules $\YD{\kk\Z_2}$ (see Lemma~\ref{prp:braided-mono-subcat}),
we may consider the bosonization
\[
\widehat{\HH} := \HH\biprod \kk\Z_2
\]
of $\HH$ by $\kk\Z_2$.
The Hopf algebra structure of $\widehat{\HH}$ is explicitly given as follows.
\begin{eqnarray*}
(h\biprod \sigmab^i)(h'\biprod \sigmab^j) &=& (-1)^{i|h'|} h h' \biprod  \sigmab^{i+j}, \\
1_{\widehat{\HH}} &=& 1_\HH\biprod  \eb, \\
\Delta_{\widehat{\HH}}(h\biprod \sigmab^i) &=& h_{(1)}\biprod \sigmab^{i+|h_{(2)}|}\otimes h_{(2)} \biprod \sigmab^i, \\
\varepsilon_{\widehat{\HH}}(h\biprod \sigmab^i) &=& \varepsilon_\HH(h), \\
S_{\widehat{\HH}}(h\biprod \sigmab^i) &=& (-1)^{(i+|h|)|h|}S_\HH(h)\biprod \sigmab^{i+|h|}
\end{eqnarray*}
for $h,h'\in\HH$ and $i,j\in\{0,1\}$.
One easily sees that the map
\[
\G(\HH)\times \Z_2 \longrightarrow \G(\widehat{\HH}); \quad
(g,\sigmab^i) \longmapsto g\biprod \sigmab^i
\]
is a group isomorphism.

\begin{example}
The bosonization of the exterior superalgebra $\bigwedge(z_1,\dots,z_\theta)$
is isomorphic to the following Hopf algebra.
\[
\kk\langle c,x_i \mid c^2=1, x_ix_j=-x_jx_i, cx_i=-x_ic \; (i,j\in\{1,\dots,\theta\}) \rangle,
\]
where $c$ is group-like and $x_i$ is $c$-skew primitive $(i\in\{1,\dots,\theta\})$.
In particular, the bosonization of $\bigwedge(z)$ is isomorphic to Sweedler's four-dimensional Hopf algebra $H_4$
(cf.~Theorem~\ref{prp:2-dim}).
\qedheree
\end{example}

Any Hopf algebra $H$ can be naturally regarded as a Hopf superalgebra by letting $H_{\bar0}:=H$ and $H_{\bar1}:=0$.
Thus, in the classification theory of Hopf superalgebras, we are interested in {\it non}-purely even Hopf superalgebras,
that is, Hopf superalgebras whose odd parts are non-zero.
The following observation is useful in the classification theory of Hopf superalgebras (see Theorem~\ref{thm:main1}).
\begin{lemma} \label{prp:boson=non-triv}
If $\HH_{\bar1}\neq0$, then $\widehat{\HH}$ is neither commutative nor cocommutative.
\end{lemma}
\begin{proof}
By the assumption, we can take $0\neq x\in \HH_{\bar1}$.
Then we have
\[
(1_\HH\biprod \sigmab)(x\biprod \eb) = -x\biprod \sigmab \neq x \biprod  \sigmab = (x\biprod \eb)(1_\HH\biprod \sigmab),
\]
and hence $\widehat{\HH}$ is not commutative.
Let $\pi:\widehat{\HH}\to\kk\Z_2$ be the canonical projection given by $\pi(h\biprod\sigmab^i)=\varepsilon_\HH(h)\sigmab^i$.
Since $x$ is odd, we have
\[
\begin{aligned}
(\pi\otimes\id)\Delta_{\widehat{\HH}}(x\biprod\eb)&=\sigmab\otimes(x\biprod\eb),\\
(\id\otimes\pi)\Delta_{\widehat{\HH}}(x\biprod\eb)&=(x\biprod\eb)\otimes\eb.
\end{aligned}
\]
Since $\sigmab\neq\eb$ and $x\neq0$, these identities show that $\widehat{\HH}$ is not cocommutative.
\end{proof}

For a left $\HH$-supermodule $V$, the action
\[
\widehat{\HH}\otimes V\longrightarrow V;
\quad
(h\biprod \sigmab^i)\otimes v\longmapsto (h\biprod \sigmab^i).v := (-1)^{i|v|} h.v
\]
makes $V$ into a left $\widehat{\HH}$-module.
This correspondence gives a category equivalence between the category $\lsmod{\HH}$ of left $\HH$-supermodules and
the category $\lmod{\widehat{\HH}}$ of left $\widehat{\HH}$-modules.
Dually, for a right $\HH$-supercomodule $V$, the coaction
\[
V\longrightarrow V\otimes \widehat{\HH};
\quad
v\longmapsto v_{(0)} \otimes (v_{(1)}\biprod  \sigmab^{|v|})
\]
makes $V$ into a right $\widehat{\HH}$-comodule,
where the original right $\HH$-comodule structure on $V$ is written as $V\to V\otimes \HH; v\mapsto v_{(0)}\otimes v_{(1)}$.
This gives a category equivalence between the category $\rscomod{\HH}$ of right $\HH$-supercomodules and
the category $\rcomod{\widehat{\HH}}$ of right $\widehat{\HH}$-comodules (see also Masuoka~\cite{Mas12}).

The above equivalences preserve the underlying vector spaces and give the following result.
\begin{lemma} \label{prp:pt/ss}
$\HH$ is semisimple (resp.~pointed) if and only if $\widehat{\HH}$ is semisimple (resp.~pointed).
\end{lemma}

\begin{definition} \label{def:super-form}
Let $A$ be a Hopf algebra.
We say that $\HH$ is a {\it super-form} of $A$
if $A$ is isomorphic to $\widehat{\HH}$ as a Hopf algebra.
If $\HH$ is purely even, then $\HH$ is called a {\it trivial super-form} of $A$.
\qedheree
\end{definition}

By definition, all super-forms of a semisimple (resp.~pointed) Hopf algebra are also semisimple (resp.~pointed).

\begin{example}\label{ex:A_4}
Suppose that $\kk$ is an algebraically closed field of characteristic zero,
and let $\zeta_4\in\kk$ be a fixed primitive fourth root of unity.
According to Masuoka~\cite{Mas95b} (see also Kac and Paljutkin~\cite{KacPal66}), there is a unique (up to isomorphism) semisimple Hopf algebra of dimension $8$ that is neither commutative nor cocommutative.
We denote it by $\AM_8$.
As noted in \cite[Remark~2.14]{Mas95b}, its dual $\AM_8^*$ (which is actually isomorphic to $\AM_8$) has the following presentation:
\[
\AM_8^* = \kk \langle c,s,h \mid c^2-s^2=1,sc=cs=0,h^2=1,ch=hc,sh=-hs \rangle,
\]
where
\[
\begin{gathered}
\Delta(c)=c\otimes c-s\otimes s,\quad \varepsilon(c)=1,\quad S(c)=c, \\
\Delta(s)=c\otimes s+s\otimes c,\quad \varepsilon(s)=0,\quad S(s)=s,\\
\Delta(h)=h\otimes h+hs^2\otimes h(1-c-s),\quad \varepsilon(h)=1,\quad S(h)=h(s^2+s+1).
\end{gathered}
\]
See also Section~\ref{subsubsec:Mas}.
Note that $\G(\AM_8^*)=\{1,c+\zeta_4 s,c-\zeta_4 s,c^2+s^2\}$ ($\cong \Z_2\times\Z_2$).
One sees that the algebra map
\[
\AM_8^* \longrightarrow \widehat{\A_4(\zeta_4)};
\quad
c\mapsto x^2\biprod \sigmab,\;\;
s\mapsto -\zeta_4z^2\biprod \sigmab,\;\;
h\mapsto x\biprod \eb - \zeta_4 z\biprod \sigmab
\]
is a Hopf algebra isomorphism,
where $\A_4(\zeta_4)$ is given in Examples~\ref{ex:AM}.
Thus, we conclude that $\A_4(\zeta_4)$ and $\A_4(-\zeta_4)$ are non-trivial super-forms of $\AM_8^*$ and are semisimple by Lemma~\ref{prp:pt/ss}.
\qedheree
\end{example}

Suppose that $\HH$ is finite-dimensional.
Then for the dual Hopf superalgebra $\HH^*$ of $\HH$ (see Section~\ref{subsec:duals}),
we also consider its bosonization $\widehat{\HH^*}:=\widehat{(\HH^*)}$ of $\HH^*$ by $\kk\Z_2$.
From the discussion in Section~\ref{subsec:dual-bos}, we obtain the following.

\begin{lemma} \label{prp:dual}
The bilinear map
\[
\widehat{\HH^*} \times \widehat{\HH} \longrightarrow \kk; \quad
(f\biprod \sigmab^i, h\biprod \sigmab^j) \longmapsto (-1)^{ij} f(h)
\]
is a non-degenerate Hopf pairing and $\widehat{\HH^*}$ is isomorphic to the dual $(\widehat{\HH})^*$ of $\widehat{\HH}$ as a Hopf algebra.
\end{lemma}

We summarize the above discussion and obtain the following result which is useful in the classification theory of finite-dimensional Hopf superalgebras.
\begin{proposition} \label{prp:summary}
Let $A$ be a Hopf algebra.
If $A$ has a non-trivial super-form $\HH$, then we have the following.
\begin{enumerate}
\item\label{prp:summary:(1)} If $A$ is finite-dimensional, then its dimension is even.
\item\label{prp:summary:(2)} $A$ is neither commutative nor cocommutative.
\item\label{prp:summary:(3)} $\HH$ is semisimple if and only if $A$ is semisimple, and the same holds for pointedness.
\item\label{prp:summary:(4)} The group $\G(A)$ is decomposed into a direct product with $\Z_2$,
that is, there exists a group $\Gamma$ such that $\G(A)\cong \Gamma\times \Z_2$ as groups.
\item\label{prp:summary:(5)} If $A$ is finite-dimensional, then the group $\G(A^*)$ is decomposed into a direct product with $\Z_2$.
\end{enumerate}
\end{proposition}

\subsection{Admissible data for Hopf algebras} \label{subsec:YD-data}
Let $A$ be a Hopf algebra.
We denote its group of Hopf algebra automorphisms by $\HopfAuto(A)$ and its {\it finite dual} Hopf algebra by $A^\circ$.
Recall that $A^\circ$ is given as $\bigcup_I (A/I)^*$, where $I$ runs through the cofinite ideals of $A$.
We regard $\alpha\in\G(A^\circ)$ as an algebra map $\alpha:A\to\kk$ as usual.

It is easy to see that the map
\[
\begin{array}{ccc}
\{\kk\Z_2\hookrightarrow A^\circ \text{: Hopf inclusion}\} &\to &\{\alpha\in\G(A^\circ)\mid \ord(\alpha)=2\};\\
\iota &\mapsto& \iota(\sigmab)
\end{array}
\]
is bijective,
where ``Hopf inclusion'' means an injective Hopf algebra map.
Since the assignment $I\mapsto (A/I)^*$ gives a one-to-one correspondence between
the set of all cofinite Hopf ideals of $A$
and the set of all finite-dimensional Hopf subalgebras of $A^\circ$,
we see that the map
\[
\begin{array}{ccc}
\{A\twoheadrightarrow (\kk\Z_2)^* \text{: Hopf surjection}\} &\to &\{\kk\Z_2\hookrightarrow A^\circ \text{: Hopf inclusion}\};\\
\pi &\mapsto& {\kk\Z_2\cong (\kk\Z_2)^{**} \cong (A/\Ker(\pi))^* \hookrightarrow A^\circ}
\end{array}
\]
is bijective,
where ``Hopf surjection'' means a surjective Hopf algebra map.

Under the identification $\kk\Z_2\cong (\kk\Z_2)^*$ given in Lemma~\ref{prp:kZ_2},
for a given order-two element $\alpha\in\G(A^\circ)$,
the corresponding surjective Hopf algebra map is
\begin{equation} \label{eq:pi}
\pi : A\longrightarrow \kk\Z_2;
\quad
a\longmapsto \frac{\varepsilon_A(a)}{2}(\eb+\sigmab) + \frac{\alpha(a)}{2}(\eb-\sigmab).
\end{equation}
On the other hand, for a given order-two element $g\in \G(A)$,
the corresponding injective Hopf algebra map is
\begin{equation} \label{eq:iota}
\iota: \kk\Z_2 \longrightarrow A;\quad
\sigmab^i \longmapsto g^i.
\end{equation}
One easily sees that
$\alpha(g)=-1$ if and only if
$\pi\circ\iota = \id_{\kk\Z_2}$ (i.e., $\pi$ is a Hopf algebra split epimorphism with section $\iota$).
With this discussion in mind, we introduce the following notation.
\begin{definition} \label{def:yddata}
A pair $(g,\alpha)\in\G(A)\times \G(A^\circ)$ is called an {\it admissible datum} for $A$
if it satisfies $\ord(g)=2$, $\ord(\alpha)=2$ and $\alpha(g)=-1$.
The set of all admissible data for $A$ is denoted by $\AD{A}$.
\qedheree
\end{definition}

Let $\CC(A)$ denote the set of all pairs $(\iota,\pi)$ of a Hopf algebra split epimorphism $\pi:A\to \kk\Z_2$ with section $\iota:\kk\Z_2\to A$.
In other words, $(A;\iota,\pi)\in\T_{\kk\Z_2}$, see Section~\ref{subsec:YD}.
The above argument shows that there is a one-to-one correspondence between $\CC(A)$ and $\AD{A}$.
\begin{center}
\begin{tikzcd}
\{\kk\Z_2\hookrightarrow A:\text{Hopf inclusion}\} \times \{A\twoheadrightarrow \kk\Z_2:\text{Hopf surjection}\}
\ar[d, leftrightarrow] \ar[r,phantom, "\supset"] & \CC(A) \ar[d, leftrightarrow, dashed] \\
\{g\in \G(A)\mid \ord(g)=2\} \times \{\alpha\in\G(A^\circ)\mid\ord(\alpha)=2\} \ar[r,phantom, "\quad\;\;\supset"] & \AD{A}
\end{tikzcd}
\end{center}

For later use, for a given $(g,\alpha)\in\AD{A}$,
we define $\pi_{(g,\alpha)}:A\to \kk\Z_2$ by \eqref{eq:pi} and $\iota_{(g,\alpha)}:\kk\Z_2\to A$ by \eqref{eq:iota}.

\begin{remark}
We note that if $A$ is finite-dimensional,
then the bijection $\CC(A)\to \CC(A^*); (\iota,\pi)\mapsto(\pi^*,\iota^*)$ induces a bijection $\AD{A}\to\AD{A^*}; (g,\alpha)\mapsto (\alpha,g)$ under the canonical identification $A^{**}=A$.
\qedheree
\end{remark}

For $(g,\alpha),(h,\beta)\in \AD{A}$,
if there exists $\varphi\in\HopfAuto(A)$ such that $\varphi(g)=h$ and $\alpha=\beta\circ\varphi$,
then we write $(g,\alpha)\sim(h,\beta)$.
It is obvious that the relation $\sim$ forms an equivalence relation on $\AD{A}$.

\begin{lemma}
For $(g,\alpha),(h,\beta)\in \AD{A}$,
we have $(g,\alpha)\sim(h,\beta)$ if and only if $(\iota_{(g,\alpha)},\pi_{(g,\alpha)})\approx(\iota_{(h,\beta)},\pi_{(h,\beta)})$.
\end{lemma}
\begin{proof}
Suppose that $(g,\alpha)\sim(h,\beta)$.
Then there exists $\varphi\in\HopfAuto(A)$ such that $\varphi(g)=h$ and $\alpha=\beta\circ\varphi$.
By the definitions of $\pi$ and $\iota$, we have $\pi_{(h,\beta)}\circ\varphi=\pi_{(g,\alpha)}$ and $\varphi\circ\iota_{(g,\alpha)}=\iota_{(h,\beta)}$.
Thus, this $\varphi$ gives $(\iota_{(g,\alpha)},\pi_{(g,\alpha)})\approx(\iota_{(h,\beta)},\pi_{(h,\beta)})$.
The converse follows by the same argument.
\end{proof}

For simplicity, we put
\[
A^{\coo(g,\alpha)} := A^{\co(\iota_{(g,\alpha)},\pi_{(g,\alpha)})}
\]
for an admissible datum $(g,\alpha)\in\AD{A}$ for $A$.
Then by Proposition~\ref{prp:YD-one-to-one}, we have the following result.
\begin{proposition} \label{prp:one-to-one}
The assignment
$\AD{A}\to \YDH{\kk\Z_2};$
$(g,\alpha) \mapsto A^{\coo(g,\alpha)}$
gives a one-to-one correspondence between
$\AD{A}/\!\sim$ and
\[
\{\B \in \YDH{\kk\Z_2} \mid \B\biprod \kk\Z_2 \text{ is isomorphic to } A \text{ as a Hopf algebra} \}/\!\cong.
\]
\end{proposition}

\subsection{Super-data for Hopf algebras} \label{subsec:coinv-sub}
In the following, we also fix a Hopf algebra $A=(A;m_A,u_A,\Delta_A,\varepsilon_A,S_A)$
and take an admissible datum $(g,\alpha)\in \AD{A}$ for $A$.
Then by Proposition~\ref{prp:one-to-one},
the coinvariant subalgebra
\[
\B:=A^{\coo(g,\alpha)}
\]
of $A$ is a Hopf algebra in $\YD{\kk\Z_2}$.
In this section, we give a criterion for $\B$ to be an object of $\sVect$.
For $a\in A$, we use the following usual notations.
\[
\alpha\hits a := a_{(1)} \alpha(a_{(2)}),
\quad
a\hitted \alpha:= \alpha(a_{(1)}) a_{(2)}.
\]
Note that $\alpha\hits(-),(-)\hitted \alpha$ are algebra maps and
$\alpha \hits a \hitted \alpha := (\alpha\hits a) \hitted \alpha = \alpha\hits(a\hitted \alpha)$ for all $a\in A$.

\begin{proposition} \label{prp:B=hits}
We have $\B=\{ b\in A \mid b = \alpha \hits b \}$.
\end{proposition}
\begin{proof}
For $a\in A$, a direct computation shows that
\[
a_{(1)}\otimes \pi_{(g,\alpha)}(a_{(2)})
=\frac12 (a+(\alpha\hits a)) \otimes \eb + \frac12 (a-(\alpha\hits a))\otimes \sigmab.
\]
Thus, $a\in \B$ if and only if $a= \alpha\hits a$.
\end{proof}

The left $\kk\Z_2$-action on $\B$ is explicitly given by
\[
\kk\Z_2\otimes \B\longrightarrow \B;
\quad
\sigmab^i\otimes b \longmapsto \sigmab^i\triangleright b = g^ibg^i,
\]
see Proposition~\ref{prp:coinv-sub}.
The left $\kk\Z_2$-coaction on $\B$ can be rephrased as follows.
\begin{lemma}
For $a\in A$ and $i\in\{0,1\}$,
we have $(\pi_{(g,\alpha)}\otimes \id_A)\Delta_A(a)=\sigmab^i\otimes a$
if and only if $a\hitted \alpha = (-1)^i a$.
\end{lemma}

Thus, if we let
\[
\B_{i,j}:=\{b\in \B \mid b\hitted\alpha=(-1)^i b \text{ and } gbg=(-1)^jb \}
\quad(i,j\in\{0,1\}),
\]
then $\B$ decomposes as $\B=\bigoplus_{i,j\in\{0,1\}} \B_{i,j}$.
The following is a criterion for $\B$ to be a Hopf superalgebra such that $\widehat{\B} \cong A$.
\begin{proposition} \label{prp:super-criteria1}
For the coinvariant subalgebra $\B=A^{\coo(g,\alpha)}$ of $A$,
the following assertions are equivalent.
\begin{enumerate}
\item\label{prp:super-criteria1:(1)} $\B$ is a super-form of $A$.
\item\label{prp:super-criteria1:(2)} $b\hitted \alpha=gbg$ for all $b\in\B$.
\item\label{prp:super-criteria1:(3)} $\alpha\hits a \hitted\alpha = gag$ for all $a\in A$.
\end{enumerate}
\end{proposition}
\begin{proof}
The decomposition of $\B$ shows that $\B\in\sVect$ if and only if $\B_{0,1}=\B_{1,0}=0$,
and hence the conditions \eqref{prp:super-criteria1:(1)} and \eqref{prp:super-criteria1:(2)} are equivalent.
By Proposition~\ref{prp:B=hits}, we see that
\[
\alpha\hits(bg^i)\hitted\alpha = (b\hitted \alpha)g^i
\quad\text{and}\quad
g(bg^i)g = (gbg)g^i
\]
for all $b\in\B,i\in\{0,1\}$.
Since $\B\biprod \kk\Z_2\to A; b\biprod  \sigmab^i\mapsto bg^i$ is an isomorphism (see \eqref{eq:four}),
the conditions \eqref{prp:super-criteria1:(2)} and \eqref{prp:super-criteria1:(3)} are equivalent.
The proof is done.
\end{proof}

If $\B=A^{\coo(g,\alpha)}$ is a super-form of $A$,
then $\B=\B_{\bar0}\oplus\B_{\bar1}$ with
\begin{equation} \label{eq:B-parity}
\B_{\bar\epsilon} = \B_{\epsilon,\epsilon} = \{ b\in \B \mid gbg=(-1)^\epsilon b \}
\end{equation}
for $\epsilon\in\{0,1\}$.
This implies the following criterion.
\begin{proposition} \label{prp:super-criteria2}
If $\B=A^{\coo(g,\alpha)}$ satisfies one of the conditions given in Proposition~\ref{prp:super-criteria1},
then the following assertions are equivalent.
\begin{enumerate}
\item\label{prp:super-criteria2:(1)} The super-form $\B$ of $A$ is non-trivial.
\item\label{prp:super-criteria2:(2)} There exists $b\in\B$ such that $b\neq0$ and $b\hitted \alpha =-b$.
\item\label{prp:super-criteria2:(3)} $g\notin Z(A)$, where $Z(A)$ is the center of $A$.
\end{enumerate}
\end{proposition}

\begin{definition} \label{def:SD}
An admissible datum $(g,\alpha) \in \AD{A}$ for $A$ is called a {\it super-datum} for $A$ if it satisfies
$g\notin Z(A)$ and $\alpha\hits a \hitted \alpha = gag$ for all $a\in A$.
The set of all super-data for $A$ is denoted by $\SD{A}$.
\qedheree
\end{definition}

Then by definition,
the assignment $(g,\alpha) \mapsto A^{\coo(g,\alpha)}$ induces a bijection
\[
\SD{A}/\!\sim \,\,\longrightarrow \{\HH\text{: Hopf superalgebra} \mid \HH_{\bar1}\neq0 \text{ and }\widehat{\HH}\cong A \}/\!\cong.
\]
Let $(g,\alpha)\in\SD{A}$.
For $i\in\{0,1\}$ and $a\in A$, we put
\[
\Delta_A^i(a) := \frac12 a_{(1)}(1_A + (-1)^i g) \otimes a_{(2)}.
\]
The comultiplication and the antipode of the Hopf superalgebra structure can be written down as follows (cf.~\eqref{eq:AEG-Delta-S}).
\begin{proposition} \label{prp:Delta-S}
Let $\epsilon\in\{0,1\}$, and let $\B:=A^{\coo(g,\alpha)}$.
For $b\in \B_{\bar\epsilon}$, we get $\Delta_\B(b) = \Delta_A^0(b) - (-1)^\epsilon (g\otimes 1_A)\Delta_A^1(b)$
and $S_\B(b)= g^{\epsilon} S_A(b)$.
\end{proposition}
\begin{proof}
Since the antipode of $\kk\Z_2$ is the identity, we get
\[
\Delta_\B(b)
= b_{(1)} (\iota_{(g,\alpha)}\circ\pi_{(g,\alpha)})(b_{(2)}) \otimes b_{(3)}
= \Delta_A^0(b) + \frac12 b_{(1)} (1_A-g) \otimes (b_{(2)}\hitted \alpha),
\]
where we write $\Delta_A(b)=b_{(1)}\otimes b_{(2)}$.
By Proposition~\ref{prp:B=hits}, we have $\Delta_A(b) = \Delta_A(\alpha\hits b) = b_{(1)} \otimes (\alpha\hits b_{(2)})$.
Combining this equation with Proposition~\ref{prp:super-criteria1}\eqref{prp:super-criteria1:(2)},
we get
\begin{equation} \label{eq:Delta-S:(1)}
b_{(1)}\otimes (b_{(2)}\hitted \alpha) = b_{(1)} \otimes gb_{(2)}g.
\end{equation}
On the other hand, by \eqref{eq:B-parity}, we get $\Delta_A(b) = (-1)^\epsilon \Delta_A(gbg)$,
and hence we have
\begin{equation} \label{eq:Delta-S:(2)}
gb_{(1)}g \otimes b_{(2)} = (-1)^\epsilon b_{(1)}\otimes gb_{(2)}g.
\end{equation}
Thus, we obtain
\begin{eqnarray*}
\frac12 b_{(1)} (1_A-g) \otimes (b_{(2)}\hitted \alpha)
&\overset{\eqref{eq:Delta-S:(1)}}=& \frac12 b_{(1)} (1_A-g) \otimes  gb_{(2)}g \\
&\overset{\eqref{eq:Delta-S:(2)}}=& -(-1)^\epsilon \frac12 gb_{(1)} (1_A-g) \otimes b_{(2)}.
\end{eqnarray*}
This proves the first claim.
The second claim follows directly.
\end{proof}

For the reader's convenience, we give a summary of our results as follows.
\begin{theorem} \label{thm:super-form}
Let $(g,\alpha)\in\SD{A}$, and let $\HH:=A^{\coo(g,\alpha)}$ be the coinvariant subalgebra of $A$ with respect to $(g,\alpha)$.
Then
\[
\HH=\{ b\in A \mid \alpha\hits b = b\},\quad
\HH_{\bar\epsilon} = \{ b\in \HH \mid gbg=(-1)^{\epsilon} b \}
\quad(\epsilon\in\{0,1\})
\]
and the Hopf superalgebra structure of $\HH$ is given as follows.
\begin{itemize}
\item(comultiplication) $\Delta_\HH(b) = \Delta_A^0(b) - (-1)^{|b|} (g\otimes 1_A)\Delta_A^1(b)$ for $b\in\HH$.
\item(counit) $\varepsilon_\HH = \varepsilon_A$.
\item(antipode) $S_\HH(b)= g^{|b|} S_A(b)$ for $b\in\HH$.
\end{itemize}
\end{theorem}

\section{Classification of some classes of Hopf superalgebras} \label{sec:class}
For the rest of this paper, we assume that $\kk$ is algebraically closed of characteristic zero.
In this section,
we classify Hopf superalgebras of dimension up to $5$ (Sections~\ref{subsec:general}, \ref{subsec:pt-4-dim} and \ref{subsec:Mas}).
We also determine semisimple Hopf superalgebras of dimension $6$ (Section~\ref{subsec:Fuk})
and give some examples of non-semisimple non-pointed Hopf superalgebras of dimension $8$ (Section~\ref{subsec:CDMM}).

\subsection{Hopf superalgebras of prime dimension} \label{subsec:general}
We determine Hopf superalgebras of prime dimension.
Two-dimensional Hopf superalgebras are given as follows.
\begin{theorem} \label{prp:2-dim}
Up to isomorphism, the exterior superalgebra $\bigwedge(z)$ is the only two-dimensional Hopf superalgebra whose odd part is non-zero.
\end{theorem}
\begin{proof}
It is known that (up to isomorphism) Sweedler's four-dimensional Hopf algebra $H_4$ (see \eqref{eq:Sweedler} for the definition)
is the only Hopf algebra of dimension $4(=2\times2)$ which is neither commutative nor cocommutative.
Since $x^2=0$, we see that $\alpha\in\G(H_4^*)$ defined by $\alpha(c)=-1,\alpha(x)=0$ is the only element in $\G(H_4^*)$ of order two,
and hence we have $\AD{H_4}=\{(c,\alpha)\}$.
Since $c\not\in Z(H_4)$ and $\alpha\hits x \hitted \alpha= -x = cxc$,
we get $\SD{H_4}=\{(c,\alpha)\}$.
Therefore, the coinvariant subalgebra $\HH:=H_4^{\coo(c,\alpha)}$ of $H_4$ is generated by $x$ and becomes a non-trivial super-form of $H_4$.
The Hopf superalgebra structure of $\HH$ is calculated as
\[
\Delta_{\HH}(x) = 1\otimes x+ x\otimes 1,\;\;
\varepsilon_{\HH}(x) = 0,\;\;
S_{\HH}(x) = -c^2x = -x
\]
by Theorem~\ref{thm:super-form}.
Thus, the assignment $z\mapsto x$ gives a Hopf superalgebra isomorphism $\bigwedge(z)\cong\HH$.
\end{proof}

We next show that every Hopf superalgebra of odd prime dimension is purely even.
\begin{theorem} \label{thm:main1}
All Hopf superalgebras of odd prime dimensions are purely even.
\end{theorem}
\begin{proof}
Let $p$ be an odd prime number, and let $\HH$ be a $p$-dimensional Hopf superalgebra.
Since the bosonization $\widehat{\HH}$ of $\HH$ is a Hopf algebra of dimension $2p$,
by Masuoka~\cite{Mas95} and Ng~\cite{Ng05}, there exists a finite group $\Gamma$ such that $\widehat{\HH}$ is isomorphic to
the group algebra $\kk \Gamma$ over $\Gamma$ or its dual $(\kk \Gamma)^*$.
Hence by Lemma~\ref{prp:boson=non-triv}, the odd part $\HH_{\bar1}$ of $\HH$ must be zero.
\end{proof}

\subsection{Pointed Hopf superalgebras of dimension $4$} \label{subsec:pt-4-dim}
In this and the following subsection, we classify four-dimensional Hopf superalgebras and study their properties using our method.
Together with Theorem~\ref{thm:main1}, this gives a complete classification of Hopf superalgebras of dimension up to $5$.
By the classification results of Masuoka~\cite{Mas95b} and \c{S}tefan~\cite{Ste99}, every Hopf algebra of dimension $8$ is semisimple or pointed, with one exception.
The dual of this exceptional Hopf algebra is pointed with group of group-like elements isomorphic to $\Z_4$, as described by Garc\'ia and Vay~\cite[Section~3.1]{GarVay10}.
By Proposition~\ref{prp:summary}\eqref{prp:summary:(5)}, this exceptional Hopf algebra has no non-trivial super-form.
Thus it remains to consider the semisimple and pointed cases.
Here we determine the super-forms in the pointed case.

By Proposition~\ref{prp:summary} and the classification of non-semisimple pointed Hopf algebras of dimension $8$ given in~\cite{Ste99}, all such Hopf algebras except the following two have no non-trivial super-forms.
\begin{itemize}
\item $A_{C_2}=\kk\langle c,x,y \mid c^2=1, cx=-xc, cy=-yc, xy=-yx, x^2=y^2=0 \rangle$,
where $c$ is group-like and $x,y$ are $c$-skew primitive.
\item $A_{C_2\times C_2} = \kk \langle c,d,x \mid c^2=d^2=1, cd=dc, cx=-xc, dx=-xd, x^2=0\rangle$,
where $d,c$ are group-like and $x$ is $c$-skew primitive.
\end{itemize}

Super-forms of $A_{C_2}$ are given as follows.
\begin{proposition}
Up to isomorphism, the exterior superalgebra $\HH_4^{(1)}=\bigwedge(z_1,z_2)$ is the only Hopf superalgebra whose bosonization is isomorphic to $A_{C_2}$.
\end{proposition}
\begin{proof}
It is easy to see that $\alpha\in\G(A_{C_2}^*)$ defined by $\alpha(c)=-1,\alpha(x)=\alpha(y)=0$ is the only algebra map $A_{C_2}\to\kk$ of order two.
Since $\alpha\hits x\hitted \alpha = -x=cxc$ and $\alpha\hits y\hitted \alpha=-y=cyc$,
we obtain $\AD{A_{C_2}}=\SD{A_{C_2}}=\{(c,\alpha)\}$.
By Proposition~\ref{prp:B=hits}, one sees that
the coinvariant subalgebra $A_{C_2}^{\coo(c,\alpha)}$ of $A_{C_2}$ is generated by $x,y$.
By Theorem~\ref{thm:super-form}, $x$ and $y$ are odd primitive.
Thus, the assignment $z_1\mapsto x,z_2\mapsto y$ gives a Hopf superalgebra isomorphism $\bigwedge(z_1,z_2)\cong A_{C_2}^{\coo(c,\alpha)}$.
\end{proof}

We next find super-forms of $A_{C_2\times C_2}$.
It is easy to see that
\[
\G(A_{C_2 \times C_2}) = \{ 1, c, d, cd \}
\quad \text{and} \quad
\G(A_{C_2 \times C_2}^*) = \{ \varepsilon, \alpha_1, \alpha_2, \alpha_3 := \alpha_1 \alpha_2 \},
\]
where $\alpha_1$ and $\alpha_2$ are algebra maps $A_{C_2 \times C_2} \to \kk$ determined by
$\alpha_1(c)=-1$, $\alpha_1(d)=1$, $\alpha_2(c)=1$, $\alpha_2(d)=-1$ and $\alpha_1(x) = \alpha_2(x)=0$. Hence,
\[ \AD{A_{C_2\times C_2}}=\{ (c,\alpha_1), (cd,\alpha_1), (d,\alpha_2), (cd,\alpha_2), (c,\alpha_3), (d,\alpha_3) \}. \]
Since $cd$ is central, we obtain
\[
\SD{A_{C_2\times C_2}} = \{ (c,\alpha_1), (c,\alpha_3), (d,\alpha_3) \}.
\]
The group $\HopfAuto(A_{C_2 \times C_2})$ of Hopf algebra automorphisms of $A_{C_2 \times C_2}$ is isomorphic to $\kk^{\times}$.
More precisely, we have
\[
\HopfAuto(A_{C_2 \times C_2}) = \{ \varphi_u \mid u \in \kk^{\times} \},
\]
where $\varphi_u$ is the algebra automorphism of $A_{C_2\times C_2}$ determined by $\varphi_u|_{\G(A_{C_2 \times C_2})} = \id$ and $\varphi_u(x)=ux$.
We obtain the following result.
\begin{proposition}
Up to isomorphism, Hopf superalgebras $\HH_4^{(2)}$, $\HH_4^{(3)}$ and $\HH_4^{(4)}$ of Example~\ref{ex:AM} are the only ones whose bosonization is isomorphic to $A_{C_2\times C_2}$.
\end{proposition}
\begin{proof}
By the above discussion, the three elements of $\SD{A_{C_2 \times C_2}}$ are pairwise non-equivalent.
We compute the coinvariant subalgebra for each element.
First, we consider the coinvariant subalgebra $A_{C_2\times C_2}^{\coo(c,\alpha_1)}$ of $A_{C_2\times C_2}$ with respect to $(c,\alpha_1)$.
By Proposition~\ref{prp:B=hits}, we see that
$A_{C_2\times C_2}^{\coo(c,\alpha_1)}$ is generated by $d,x$.
By Theorem~\ref{thm:super-form}, $d$ is group-like and $x$ is odd primitive.
Thus, the assignment $g\mapsto d,z\mapsto x$ gives a Hopf superalgebra isomorphism $\HH_4^{(4)}\cong A_{C_2\times C_2}^{\coo(c,\alpha_1)}$.

Next, we consider the coinvariant subalgebra $A_{C_2\times C_2}^{\coo(c,\alpha_3)}$ of $A_{C_2\times C_2}$ with respect to $(c,\alpha_3)$.
By Proposition~\ref{prp:B=hits} and Theorem~\ref{thm:super-form},
we see that $A_{C_2\times C_2}^{\coo(c,\alpha_3)}$ is generated by $cd,x$, where $cd$ is group-like and $x$ is odd primitive.
Thus, the assignment $g\mapsto cd,z\mapsto x$ gives a Hopf superalgebra isomorphism $\HH_4^{(2)}\cong A_{C_2\times C_2}^{\coo(c,\alpha_3)}$.

Finally, we consider the coinvariant subalgebra $A_{C_2\times C_2}^{\coo(d,\alpha_3)}$ of $A_{C_2\times C_2}$ with respect to $(d,\alpha_3)$.
We see that $A_{C_2\times C_2}^{\coo(d,\alpha_3)}$ is generated by $cd,x$, where $cd$ is group-like and $x$ is odd $cd$-skew primitive.
Thus, the assignment $g\mapsto cd,z\mapsto x$ gives a Hopf superalgebra isomorphism $\HH_4^{(3)}\cong A_{C_2\times C_2}^{\coo(d,\alpha_3)}$.
\end{proof}

Thus, we obtain the following result.
\begin{theorem} \label{thm:ShiShiWak2}
Let $\HH$ be a non-semisimple pointed Hopf superalgebra of dimension $4$ with $\HH_{\bar1}\neq0$.
Then $\HH$ is isomorphic to one of the Hopf superalgebras $\HH_4^{(1)},\HH_4^{(2)},\HH_4^{(3)},\HH_4^{(4)}$,
which are pairwise non-isomorphic.
Moreover, $\HH_4^{(1)}$ and $\HH_4^{(2)}$ are self-dual and the dual of $\HH_4^{(3)}$ is isomorphic to $\HH_4^{(4)}$.
\end{theorem}

\subsection{Semisimple Hopf superalgebras of dimension $4$} \label{subsec:Mas}
We next determine super-forms of eight-dimensional semisimple Hopf algebras.
In the following, we fix a primitive fourth root of unity $\zeta_4\in\kk$.

\subsubsection{Super-forms of $\AM_8$} \label{subsubsec:Mas}
The classification of semisimple Hopf algebras of dimension $8$ was obtained by Masuoka~\cite{Mas95b}.
By his result, the following is the only {\it non-trivial} (i.e., neither commutative nor cocommutative) semisimple Hopf algebra of dimension $8$:
\[
\AM_8 := \kk\left\langle X,Y,Z \vphantom{\begin{tabular}{c}$X^2=Y^2=1,\;Z^2=\frac12(1+X+Y-XY)$,\\ $XY=YX,\;ZX=YZ,\;YZ=XZ$\end{tabular}}\;\;\right|
\left.\begin{tabular}{l}$X^2=Y^2=1,\;Z^2=\frac12(1+X+Y-XY)$,\\ $XY=YX,\;ZX=YZ,\;XZ=ZY$\end{tabular} \right\rangle,
\]
where $X$ and $Y$ are group-like and
\[
\begin{gathered}
\Delta(Z)=\frac12(Z\otimes Z + Z\otimes XZ + YZ\otimes Z - YZ\otimes XZ),\\
\varepsilon(Z) = 1,\quad S(Z) = Z.
\end{gathered}
\]
Both $\G(\AM_8)$ and $\G(\AM_8^*)$ are isomorphic to $\Z_2 \times \Z_2$ and are given by
\[
\G(\AM_8)=\{1,X,Y,XY\},\qquad
\G(\AM_8^*)=\{\varepsilon,\alpha_+,\alpha_-,\alpha_+ \alpha_-\},
\]
where $\alpha_{\s}$ for $\s \in \{ +, - \}$ is given by
$\alpha_{\s}(X) = \alpha_{\s}(Y) = -1$ and $\alpha_{\s}(Z)=\s\zeta_4$.
Hence the set of all admissible data for $\AM_8$ is given as
\[
\AD{\AM_8}=\{(X,\alpha_+),(X,\alpha_-),(Y,\alpha_+),(Y,\alpha_-)\}.
\]
We also have $\HopfAuto(\AM_8) = \{ \id, \psi, \phi, \psi \phi \}$, where
\[
\begin{gathered}
\phi(X) = X, \quad \phi(Y) = Y, \quad \phi(Z)=XYZ, \\
\psi(X) = Y, \quad \psi(Y) = X, \quad \psi(Z) = \frac{1}{2}(1+X+Y-XY)Z.
\end{gathered}
\]
See, for example, Sage and Vega~\cite[Section~4.2]{SagVeg12} and Shi and Hu~\cite{Shi23}.
A direct computation gives the following equivalence classes of admissible data for $\AM_8$.

\begin{lemma} \label{prp:sim}
We have $(X,\alpha_+)\sim(Y,\alpha_-)$ and $(X,\alpha_-)\sim(Y,\alpha_+)$.
Moreover, $(X,\alpha_+)\not\sim(X,\alpha_-)$.
\end{lemma}

To determine super-data for $\AM_8$, we first calculate the left and right actions of $\alpha_\s$ for each $\s\in\{+,-\}$ as follows.
\begin{equation} \label{eq:H_8hits}
\alpha_\s\hits X = -X,\quad
\alpha_\s\hits Y = -Y,\quad
\alpha_\s\hits Z = \s\zeta_4 YZ,
\end{equation}
\begin{equation} \label{eq:H_8hitted}
X\hitted\alpha_\s = -X,\quad
Y\hitted\alpha_\s = -Y,\quad
Z\hitted\alpha_\s = \s\zeta_4 XZ.
\end{equation}
\begin{lemma} \label{prp:SD(H_8)}
The set $\{ (X,\alpha_+), (X,\alpha_-) \}$ is a complete set of representatives of equivalence classes of $\SD{\AM_8}$.
\end{lemma}

Therefore, up to isomorphism, there are exactly two non-trivial super-forms of $\AM_8$.
By Example~\ref{ex:A_4},
we see that such Hopf superalgebras are exhausted by $\A_4(\pm\zeta_4)$ given in Example~\ref{ex:AM}.
Thus, we get the following result.
\begin{theorem} \label{thm:4ss}
Let $\HH$ be a semisimple Hopf superalgebra of dimension $4$ with $\HH_{\bar1}\neq0$.
Then $\HH$ is isomorphic to one of the Hopf superalgebras $\A_4(\zeta_4),\A_4(-\zeta_4)$.
Moreover, these Hopf superalgebras are non-isomorphic.
\end{theorem}

\subsubsection{Coinvariant subalgebras} \label{subsubsec:Mas2}
In the following, we give an explicit isomorphism between
$\A_4(-\zeta_4)$ and the coinvariant subalgebra $\B:=\AM_8^{\coo(X,\alpha_+)}$ of $\AM_8$.
To do this, we first determine generators and relations of the algebra $\B$.
Set
\begin{eqnarray*}
g &:=& XY, \\
v &:=& \frac{1-\zeta_4}{4}(Z+\zeta_4XZ+\zeta_4YZ+XYZ), \\
w &:=& \frac{\sqrt{2}}{4}(Z-\zeta_4XZ+\zeta_4YZ-XYZ).
\end{eqnarray*}
By \eqref{eq:H_8hits} and Proposition~\ref{prp:B=hits}, we see that $\B$ has the set $\{ 1, g, v, w \}$ as a basis. Moreover, $v$ and $w$ are homogeneous with $|v|=0$ and $|w|=1$, respectively.
By direct computation, we have
\[
\begin{gathered}
g^2=1,\quad gv=v,\quad gw=-w,\\
wv=0=vw,\quad v^2=\frac12(1+g),\quad w^2=\frac12(1-g).
\end{gathered}
\]
In particular, we have $v^2 - w^2 = g$.
Since $\B$ has dimension 4, it is now easy to see that $\B$ is generated by $v$ and $w$ with the relations $v^2 + w^2 = 1$ and $v w = w v = 0$.

By Theorem~\ref{thm:super-form}, one sees that the comultiplication $\Delta_\B$ of $\B$ is given as
\[
\Delta_\B(v) = v\otimes v - \zeta_4 w\otimes w,\quad \Delta_\B(w) = v\otimes w + w\otimes v.
\]
These results imply the following.
\begin{proposition}
There is an isomorphism of Hopf superalgebras determined by
\[ \A_4(-\zeta_4) \longrightarrow \B = \AM_8^{\coo(X,\alpha_+)};
\quad
x\longmapsto v,\quad z\longmapsto w.
\]
\end{proposition}

An isomorphism $\A_4(\zeta_4) \cong \AM_8^{\coo(X,\alpha_-)}$ is obtained by the same argument with $\zeta_4$ replaced by $-\zeta_4$.

\subsubsection{Duals} \label{subsubsec:Mas3}
In this section, we determine the dual of the Hopf superalgebra
\[
\HH := \A_4(-\zeta_4) = \kk\langle x,z \mid x^2+z^2=1, xz=zx=0 \rangle.
\]
Recall that $|x|=0$, $|z|=1$, $\Delta_\HH(x)=x\otimes x - \zeta_4 z\otimes z$ and $\Delta_\HH(z)=x\otimes z + z\otimes x$.

Let $1^*,x^*,(x^2)^*,z^*$ denote the basis dual to $1,x,x^2,z$.
Recall that the evaluation map $\langle\;,\; \rangle : \HH^*\times \HH\to \kk$ determines a Hopf superalgebra structure on $\HH^*$ (see Section~\ref{subsec:duals}).
The multiplication table of $\HH^*$ is given as follows.
\[
\begin{array}{c||c|c|c|c|}
    & 1^* & x^* & (x^2)^* & z^* \\
\hline\hline
1^* & 1^*+(x^2)^* & 0 & -(x^2)^* & 0 \\ \hline
x^* & 0 & x^* & 0 & z^* \\ \hline
(x^2)^* & -(x^2)^* & 0 & 2(x^2)^* & 0 \\ \hline
z^* & 0 & z^* & 0 & -\zeta_4 x^* \\ \hline
\end{array}
\]
Also, the comultiplication is
\begin{eqnarray*}
\Delta_{\HH^*}(1^*) &=& 1^*\otimes 1^* + z^*\otimes z^*,\\
\Delta_{\HH^*}(x^*) &=& 1^* \otimes x^* + x^*\otimes 1^* + x^*\otimes (x^2)^* + (x^2)^*\otimes x^*,\\
\Delta_{\HH^*}((x^2)^*) &=& 1^* \otimes (x^2)^* + (x^2)^*\otimes 1^* + x^*\otimes x^*\\
&&{}+ (x^2)^*\otimes (x^2)^* - z^*\otimes z^*,\\
\Delta_{\HH^*}(z^*) &=& 1^*\otimes z^* + z^*\otimes 1^*
\end{eqnarray*}
and the counit is
\[
\varepsilon_{\HH^*}(1^*)=1,\quad \varepsilon_{\HH^*}(x^*)=0,\quad \varepsilon_{\HH^*}((x^2)^*)=0,\quad \varepsilon_{\HH^*}(z^*)=0.
\]

Choose $c\in\kk$ such that $c^2=\zeta_4$.
A direct computation shows that the algebra map
\[
\HH\longrightarrow \HH^*;
\quad
x\longmapsto 1^*,\;\;
z\longmapsto c z^*
\]
is an isomorphism of Hopf superalgebras.
In other words, the pairing $\langle\;,\;\rangle :\HH\times \HH\to\kk$ defined by
\[
\langle x,x \rangle = \langle x,z\rangle = \langle z,x\rangle =0,
\quad
\langle z,z\rangle = c
\]
is a non-degenerate Hopf pairing.
Thus, $\HH$ is self-dual.
In this way, we get the following result.
\begin{theorem} \label{thm:A_4-self-dual}
The Hopf superalgebras $\A_4(\zeta_4)$ and $\A_4(-\zeta_4)$ are self-dual.
\end{theorem}

\subsection{Semisimple Hopf superalgebras of dimension $6$} \label{subsec:Fuk}
In this section, we first construct a semisimple Hopf superalgebra of dimension $6$ with non-zero odd part using Theorem~\ref{thm:AEG}.
We then show that it is, up to isomorphism, the only semisimple Hopf superalgebra of dimension $6$ with non-zero odd part.

\subsubsection{Existence}
By Masuoka~\cite[Theorem~1.10]{Mas95b},
any semisimple Hopf algebra of dimension $6$ is isomorphic to one of $\kk \Z_6$, $(\kk{\sym})^*$ and $\kk\sym$,
where
\[
\sym=\langle s_1,s_2 \mid s_1^2=s_2^2=e,\; s_1s_2s_1=s_2s_1s_2\rangle
\]
is the symmetric group of degree three ($e\in\sym$ is the identity element).
Note that $\kk\sym$ is the only non-commutative Hopf algebra among them and has a non-central group-like element
\[
c:=s_1\quad\in\G(\kk\sym)
\]
of order two.

Therefore, by Theorem~\ref{thm:AEG}, we make the Hopf algebra $\HH:=\kk\sym$ into a Hopf superalgebra such that $\HH_{\bar1}\neq0$.
Moreover, $\HH$ is semisimple as a Hopf superalgebra, since $\kk\sym$ is semisimple as a Hopf algebra.
In the following, we shall write down the Hopf superalgebra structure of $\HH$.
For simplicity, we set
\[
\begin{gathered}
x:=\frac12(s_1s_2+s_2s_1),\quad
y:=\frac12(s_2+s_1s_2s_1),\\
z:=\frac12(s_1s_2-s_2s_1),\quad
w:=\frac12(s_2-s_1s_2s_1).
\end{gathered}
\]
Since $cxc=x,cyc=y,czc=-z$ and $cwc=-w$,
we have
\[
\HH_{\bar0}=\kk e\oplus \kk c\oplus \kk x \oplus \kk y
\quad\text{and}\quad
\HH_{\bar1}=\kk z\oplus \kk w.
\]
The multiplication table of $\HH$ is given as follows.
\[
\begin{array}{c||c|c|c|c|c|c|}
  & e & c & x & y & z & w \\
\hline\hline
e & e & c & x & y & z & w \\ \hline
c & c & e & y & x & w & z \\ \hline
x & x & y & \frac12(e+x) & \frac12(c+y) & -\frac12z & -\frac12w\\ \hline
y & y & x & \frac12(c+y) & \frac12(e+x) & -\frac12w & -\frac12z\\ \hline
z & z & -w & -\frac12z & \frac12w & \frac12(x-e) & \frac12(c-y) \\ \hline
w & w & -z & -\frac12w & \frac12z & \frac12(y-c) & \frac12(e-x) \\ \hline
\end{array}
\]

Since $\Delta_{\kk\sym}(x)=x\otimes x + z\otimes z$,
we get $\Delta_{\HH,\bar0}(x) = x\otimes x$ and $\Delta_{\HH,\bar1}(x) = z\otimes z$ by definition.
Thus, the comultiplication $\Delta_\HH$ of $x$ in $\HH$ is given as
\[
\Delta_{\HH}(x) = \Delta_{\HH,\bar0}(x)-(-1)^{|x|}(c\otimes1)\Delta_{\HH,\bar1}(x) = x\otimes x -w\otimes z.
\]
The counit and the antipode are given as $\varepsilon_\HH(x) = \varepsilon_{\kk\sym}(x) = 1$
and $S_{\HH}(x) = (-1)^{|x|}S_{\kk\sym}(x) = x$, respectively.
In this way, we get the following result.
\begin{proposition} \label{prp:H_6^5}
Let $\A_6$ be the $6$-dimensional superalgebra generated by homogeneous elements $x,y,z$ and $w$ subject to
\[
\begin{gathered}
|x|=|y| =0,\quad |z|=|w|=1,\\
x^2=\frac12(1+x)=y^2,\quad z^2=\frac12(x-1)=-w^2,\\
xy=yx,\quad xz=-\frac12z=zx,\quad xw=-\frac12w=wx,\\
yz=-\frac12w=-zy,\quad yw=-\frac12z=-wy,\quad zw=xy-y=-wz.
\end{gathered}
\]
Then the following comultiplication $\Delta$, counit $\varepsilon$ and antipode $S$ make $\A_6$ into a Hopf superalgebra.
\[
\begin{gathered}
\Delta(x)=x\otimes x-w\otimes z,\quad
\Delta(y)=y\otimes y-z\otimes w,\\
\Delta(z)=z\otimes x+y\otimes z,\quad
\Delta(w)=w\otimes y+x\otimes w,\\
\varepsilon(x)=\varepsilon(y)=1,\quad \varepsilon(z)=\varepsilon(w)=0,\\
S(x)=x,\quad S(y)=y,\quad S(z)=-w,\quad S(w)=z.
\end{gathered}
\]
Moreover, $\A_6$ is semisimple as a Hopf superalgebra.
\end{proposition}
The set of all group-like elements is given as $\G(\A_6)=\{1,xy+zw\}\;(\cong\Z_2)$.

\subsubsection{Uniqueness}
In this section, we show the following uniqueness result.
\begin{theorem} \label{thm:H_6^5}
Let $\HH$ be a semisimple Hopf superalgebra with non-zero odd part and $\dim\HH=6$.
Then $\HH$ is isomorphic to the Hopf superalgebra $\A_6$ given in Proposition~\ref{prp:H_6^5}.
Moreover, $\A_6$ is self-dual.
\end{theorem}

By Fukuda~\cite{Fuk97}, any semisimple Hopf algebra of dimension $12$ is isomorphic to one of $\kk \Gamma$, $(\kk{\Gamma})^*$, $A_+$ and $A_-$, where $\Gamma$ is a group of order $12$.
By \cite[Remark~4.2]{Fuk97}, we have $\G(A_+)\cong \Z_2\times \Z_2$ and $\G(A_-)\cong \Z_4$, and $A_{\pm}$ are self-dual.
By Proposition~\ref{prp:summary}, only $A_+$ among these Hopf algebras can have a non-trivial super-form.

Let us recall the definition of the Hopf algebra $A:=A_+$.
In the following, we set $c:=s_1\in\sym$ and denote the dual basis element corresponding to $\sigma\in\sym$ by $\sigma^*$.
As an algebra, $A$ is generated by $(\kk{\sym})^*$ and $\z$ such that
\[
\z^2 = 1,\;\;
\z f = f^c \z\;\;
\text{for all } f\in(\kk{\sym})^*,\;\;
\text{and }(\kk\sym)^*\text{ is a subalgebra of }A,
\]
where $f^c\in \kk^{\sym}$ is defined by $f^c(\sigma) := f(c\sigma c)$ for $\sigma\in\sym$.
The Hopf algebra structure of $A$ is given as
\[
\Delta_A(\sigma^*) = \sum_{\tau\tau'=\sigma} \tau^*\otimes \tau'^*,
\quad
\varepsilon_A(\sigma^*) = \delta_{\sigma,e},
\quad
S_A(\sigma^*) = (\sigma^{-1})^*
\]
and $\Delta_A(\z)=\z\otimes \z, \varepsilon_A(\z)=1, S_A(\z)=\z$ (i.e., $\z$ is group-like),
where $\delta_{\sigma,e}$ is the Kronecker symbol and $e\in\sym$ is the identity element (as before).

Let $\sgn\in(\kk{\sym})^*$ be the signature map of $\sym$.
Then we have
\[
\G(A) = \{1,\z,\sgn,\z\,\sgn\},\qquad
\G(A^*) = \{\varepsilon,\alpha_1,\alpha_2,\alpha_3:=\alpha_1\alpha_2\},
\]
where $\alpha_1$ and $\alpha_2$ are given by $\alpha_1(\z)=1$, $\alpha_1(f)=f(c)$, $\alpha_2(\z)=-1$ and $\alpha_2(f)=f(e)$ for $f\in(\kk\sym)^*$.

The set of all admissible data for $A$ is given as follows.
\[
\AD{A} = \{ (\z,\alpha_2), (\z,\alpha_3), (\sgn,\alpha_1), (\sgn,\alpha_3), (\z\,\sgn,\alpha_1), (\z\,\sgn,\alpha_2) \}.
\]
\begin{lemma}
We have $(\z,\alpha_2)\sim(\z\,\sgn,\alpha_2)$ and $(\z,\alpha_3)\sim(\z\,\sgn,\alpha_1)$.
\end{lemma}
\begin{proof}
An algebra map $\varphi:A\to A$ determined by
\[
\varphi|_{(\kk{\sym})^*}=\id \quad\text{and}\quad \varphi(\z)=\z\,\sgn
\]
is a Hopf algebra automorphism of $A$.
Using this $\varphi$, the claim easily follows.
\end{proof}

Using this lemma, we show that $\SD{A}/\!\sim$ consists of a single element.
\begin{lemma} \label{prp:SD(A_+)}
We have $\SD{A}/\!\sim\; = \{ [(\z,\alpha_3)] \}$.
\end{lemma}
\begin{proof}
Since the signature map $\sgn$ is central, we have $(\sgn,\alpha_1), (\sgn,\alpha_3) \notin \SD{A}$ by definition.
Let us consider the case of $(\z,\alpha_2)$.
For all $\sigma\in\sym$, we get
\[
\alpha_2\hits \z = -\z,\quad
\alpha_2 \hits \sigma^* = \sum_{\tau\tau'=\sigma} \tau^*\; \delta_{\tau',e} = \sigma^*, \quad
\sigma^* \hitted \alpha_2 = \sigma^*.
\]
This shows that the coinvariant subalgebra $A^{\coo(\z,\alpha_2)}$ of $A$ coincides with $(\kk{\sym})^*$.
However, for $s_2\in\sym$, we have
\[
\z s_2^* \z = (s_2^*)^c\; \z^2 = (s_2^*)^c = (c s_2 c)^* = (s_1s_2s_1)^*
\neq s_2^* = \alpha_2\hits s_2^* \hitted \alpha_2,
\]
and hence we conclude that $(\z,\alpha_2) \notin \SD{A}$.

Next, we consider the case of $(\z,\alpha_3)$.
For all $\sigma\in\sym$, we get
\begin{equation} \label{eq:H_6hits}
\alpha_3\hits \z = -\z,\quad
\alpha_3\hits \sigma^* = \sum_{\tau\tau'=\sigma} \tau^* (\tau'^*(c)) = (\sigma c)^*, \quad
\sigma^* \hitted \alpha_3 = (c\sigma)^*.
\end{equation}
Thus, for all $\sigma\in\sym$, we have
\[
\z \sigma^* \z = (\sigma^*)^c\;\z^2 = (c\sigma c)^* = \alpha_3\hits \sigma^*\hitted \alpha_3.
\]
\end{proof}

The uniqueness assertion of Theorem~\ref{thm:H_6^5} follows from Proposition~\ref{prp:H_6^5} and Lemma~\ref{prp:SD(A_+)}.
Since $A_+$ is self-dual, Lemma~\ref{prp:dual} shows that $\A_6^*$ is also a non-trivial super-form of $A_+$.
Thus $\A_6^*\cong\A_6$, which completes the proof.

\subsubsection{Coinvariant subalgebras}
Let $\B$ be the coinvariant subalgebra $A^{\coo(\z,\alpha_3)}$ of $A$.
Since $\B$ is a semisimple Hopf superalgebra of dimension $6$ with non-zero odd part, Theorem~\ref{thm:H_6^5} gives $\B\cong\A_6$.
In this section, we construct a Hopf superalgebra isomorphism from $\A_6$ to $\B$ explicitly.

Set
\begin{eqnarray*}
x_1 &:=& e^*+s_1^*, \\
x_2 &:=& (e^*-s_1^*)\z, \\
x_3 &:=& s_2^*+(s_1s_2)^*+(s_2s_1)^*+(s_1s_2s_1)^*, \\
x_4 &:=& (s_2^*-(s_1s_2)^*-(s_2s_1)^*+(s_1s_2s_1)^*)\z, \\
w_1 &:=& s_2^*-(s_1s_2)^*+(s_2s_1)^*-(s_1s_2s_1)^*, \\
w_2 &:=& (s_2^*+(s_1s_2)^*-(s_2s_1)^*-(s_1s_2s_1)^*)\z.
\end{eqnarray*}
Then by \eqref{eq:H_6hits}, one sees that
$\{x_1,x_2,x_3,x_4\}$ forms a basis of $\B_{\bar0}$ and $\{w_1,w_2\}$ forms a basis of $\B_{\bar1}$.

Let $M(1|1)$ be the set of all $2\times2$ square matrices with entries in $\kk$.
With the usual matrix multiplication and the following $\Z_2$-grading, this $M(1|1)$ becomes a superalgebra.
\[
M(1|1)_{\bar0} := \{ \left(\begin{array}{c|c}a & 0 \\ \hline 0 & d \end{array}\right) \mid a,d\in\kk \},
\quad
M(1|1)_{\bar1} := \{ \left(\begin{array}{c|c}0 & b \\ \hline c & 0 \end{array}\right) \mid b,c\in\kk \}.
\]
Then one sees that
the linear map $\B \to \kk^2\oplus M(1|1)$ defined by
\[
\begin{gathered}
x_1 \longmapsto (1,1,O),\quad
x_2 \longmapsto (1,-1,O),\\
x_3 \longmapsto (0,0,\left(\begin{array}{c|c}1 & 0 \\ \hline 0 & 1 \end{array}\right)),\quad
x_4 \longmapsto (0,0,\left(\begin{array}{c|c}1 & 0 \\ \hline 0 & -1 \end{array}\right)),\\
w_1 \longmapsto (0,0,\left(\begin{array}{c|c}0 & 1 \\ \hline 1 & 0 \end{array}\right)),\quad
w_2 \longmapsto (0,0,\left(\begin{array}{c|c}0 & -1 \\ \hline 1 & 0 \end{array}\right)),
\end{gathered}
\]
is a superalgebra isomorphism.
Moreover, one can show that the following holds.
\begin{proposition}
Choose $\lambda\in\kk$ such that $\lambda^2=-3/4$.
The assignment
\[
x\longmapsto x_1-\frac12x_3,\quad
y\longmapsto x_2+\frac12x_4,\quad
z\longmapsto \lambda w_1,\quad
w\longmapsto \lambda w_2
\]
gives a Hopf superalgebra isomorphism $\A_6 \cong \B$.
\end{proposition}

\subsection{Non-semisimple non-pointed Hopf superalgebras of dimension $8$} \label{subsec:CDMM}
In this section, we fix a primitive fourth root of unity $\zeta_4\in\kk$.
By C\u{a}linescu, D\u{a}sc\u{a}lescu, Masuoka and Menini~\cite{CalDasMasMen04}, and Garc\'ia and Vay~\cite{GarVay10}, there are precisely two isomorphism classes of non-semisimple Hopf algebras of dimension $16$ such that neither they nor their duals are pointed.
They are represented by the self-dual Hopf algebras $H_{16}(\pm\zeta_4)$.
In the following, we determine (up to isomorphism) all super-forms of $H_{16}(\pm\zeta_4)$.

We concentrate on $H_{16} := H_{16}(\zeta_4)$ for simplicity.
First of all, we recall the definition of $H_{16}$.
As an algebra, $H_{16}$ is generated by the subalgebra $H_8$ (for the definition and notations, see Section~\ref{subsec:Mas}) and $T$ subject to
\[
T^2=0,\quad
TX=-XT,\quad
TY=-YT,\quad
TZ=\zeta_4XZT.
\]
The algebra $H_8$ is a Hopf subalgebra of $H_{16}$, and $T$ is an $X$-skew primitive element of $H_{16}$.

Both $\G(H_{16})$ and $\G(H_{16}^*)$ are isomorphic to $\Z_2 \times \Z_2$ and are given by
\[
\G(H_{16})=\{1,X,Y,XY\}\quad\text{and}\quad
\G(H_{16}^*)=\{\varepsilon,\alpha_+,\alpha_-,\alpha_+ \alpha_-\},
\]
where $\alpha_{\s}$ for $\s \in \{+,-\}$ is given by
$\alpha_{\s}(X) = \alpha_{\s}(Y) = -1$, $\alpha_{\s}(Z)=\s\zeta_4$ and $\alpha_{\s}(T)=0$.
Hence, we get
\[
\AD{H_{16}}=\SD{H_{16}}=\{(X,\alpha_+),(X,\alpha_-),(Y,\alpha_+),(Y,\alpha_-)\}.
\]

In contrast to the case of $H_8$ (Lemma~\ref{prp:sim}), the number of equivalence classes of super-data for $H_{16}$ is $4$.
\begin{lemma}
The set $\{(X,\alpha_+),(X,\alpha_-),(Y,\alpha_+),(Y,\alpha_-)\}$ is a complete set of representatives of $\SD{H_{16}}/\!\sim$.
\end{lemma}
\begin{proof}
Let $\varphi$ be a Hopf algebra automorphism of $H_{16}$.
Since $H_8$ is the coradical of $H_{16}$, we see that
\[
\varphi|_{H_8} \in \HopfAuto(H_8) = \{ \id, \psi, \phi, \psi \phi \}.
\]
Here we use the notation of Section~\ref{subsubsec:Mas}.
The spaces of $X$-skew primitive and $Y$-skew primitive elements of $H_{16}$ are $\kk(1-X)\oplus\kk T$ and $\kk(1-Y)$, respectively.
Since $\psi$ and $\psi\phi$ interchange $X$ and $Y$, neither extends to a Hopf algebra automorphism of $H_{16}$.
The remaining restrictions $\id$ and $\phi$ fix $X$, $Y$ and the characters $\alpha_\pm$.
Thus the four super-data lie in distinct equivalence classes.
\end{proof}

Set $\HH^{(1)} := H_{16}^{\coo(X,\alpha_+)}$, $\HH^{(2)} := H_{16}^{\coo(X,\alpha_-)}$,
$\HH^{(3)} := H_{16}^{\coo(Y,\alpha_+)}$, $\HH^{(4)} := H_{16}^{\coo(Y,\alpha_-)}$ for simplicity.
Since $\alpha_\pm\hits T = T$ and $XTX=YTY=-T$, Theorem~\ref{thm:super-form} shows that $T$ is an odd element of $\HH^{(i)}$ for each $i\in\{1,2,3,4\}$.
Set $g:=XY$ and, for $\s\in\{+,-\}$, set
\[
\begin{gathered}
v_\s:=\frac{1-\s\zeta_4}{4}(Z+\s\zeta_4XZ+\s\zeta_4YZ+XYZ),\\
w_\s:=\frac{\sqrt2}{4}(Z-\s\zeta_4XZ+\s\zeta_4YZ-XYZ).
\end{gathered}
\]
As in Section~\ref{subsubsec:Mas2}, these elements belong to the coinvariant subalgebras associated with $\alpha_\s$.
The element $v_\s$ is even and $w_\s$ is odd.
A direct computation gives
\[
Tg=gT,\quad Tv_\s=\s v_\s T,\quad Tw_\s=-\s w_\s T.
\]
By Theorem~\ref{thm:super-form}, the comultiplication of $v_\s$ is $v_\s\otimes v_\s-\s\zeta_4w_\s\otimes w_\s$ for the datum $(X,\alpha_\s)$ and $v_\s\otimes v_\s+\s\zeta_4w_\s\otimes w_\s$ for the datum $(Y,\alpha_\s)$.
In both cases, $\Delta(w_\s)=v_\s\otimes w_\s+w_\s\otimes v_\s$.
Moreover,
\[
\begin{gathered}
\Delta_{\HH^{(1)}}(T)=\Delta_{\HH^{(2)}}(T)=1\otimes T+T\otimes1,\\
\Delta_{\HH^{(3)}}(T)=\Delta_{\HH^{(4)}}(T)=g\otimes T+T\otimes1.
\end{gathered}
\]
The same calculation applies to $H_{16}(-\zeta_4)$ and gives the following result.

\begin{proposition} \label{prp:K_8}
For $\zeta\in\{\zeta_4,-\zeta_4\}$ and $\epsilon,\eta\in\{0,1\}$, let $\mathcal{K}_8(\zeta;\epsilon,\eta)$
be the $8$-dimensional superalgebra generated by $g,v,w$ and $t$ subject to
\[
\begin{gathered}
|g|=|v|=0,\quad |w|=|t|=1,\\
g^2=1,\quad
gv=v,\quad
gw=-w,\quad
vw=wv=0,\quad
v^2=\frac12(1+g),\quad
w^2=\frac12(1-g), \\
t^2=0,\quad
tg=gt,\quad
tv=(-1)^\epsilon vt,\quad
tw=-(-1)^\epsilon wt.
\end{gathered}
\]
Then the following comultiplication $\Delta$, counit $\varepsilon$ and antipode $S$ make
$\mathcal{K}_8(\zeta;\epsilon,\eta)$ into a Hopf superalgebra.
\[
\begin{gathered}
\Delta(g)=g\otimes g,\quad
\Delta(v)=v\otimes v - \zeta w\otimes w,\\
\Delta(w)=v\otimes w+w\otimes v,\quad
\Delta(t)=g^\eta \otimes t+t\otimes 1,\\
\varepsilon(g)=\varepsilon(v)=1,\quad \varepsilon(w)=\varepsilon(t)=0,\\
S(g)=g,\quad S(v)=v,\quad S(w)=\zeta w,\quad S(t)=-g^{\eta}t.
\end{gathered}
\]
Moreover, $\mathcal{K}_8(\zeta;\epsilon,\eta)$ is non-semisimple and
neither $\mathcal{K}_8(\zeta;\epsilon,\eta)$ nor $\mathcal{K}_8(\zeta;\epsilon,\eta)^*$ is pointed.
\end{proposition}

The preceding calculation gives
\[
\begin{gathered}
\HH^{(1)}\cong\mathcal{K}_8(\zeta_4;0,0),\quad
\HH^{(2)}\cong\mathcal{K}_8(-\zeta_4;1,0),\\
\HH^{(3)}\cong\mathcal{K}_8(-\zeta_4;0,1),\quad
\HH^{(4)}\cong\mathcal{K}_8(\zeta_4;1,1).
\end{gathered}
\]
For $\zeta\in\{\zeta_4,-\zeta_4\}$ and $\epsilon,\eta\in\{0,1\}$,
we may identify $\A_4(-\zeta)$ with the Hopf sub-superalgebra of $\mathcal{K}_8(\zeta;\epsilon,\eta)$ generated by $g,v,w$.
The bosonization of $\mathcal{K}_8(\zeta;\epsilon,\eta)$ is isomorphic to $H_{16}((-1)^{\epsilon+\eta}\zeta)$.
As in Section~\ref{subsubsec:Mas3}, we can determine duals of $\mathcal{K}_8(\zeta;\epsilon,\eta)$.
For example, the following gives a non-degenerate Hopf pairing
$\langle\;,\;\rangle : \mathcal{K}_8(\zeta;0,1)\times \mathcal{K}_8(\zeta;1,0) \to \kk$.
\[
\begin{array}{c||c|c|c|c|}
\langle\;,\;\rangle & g & v & w & t \\ \hline \hline
g & 1 & -1 & 0 & 0 \\ \hline
v & -1 & 0 & 0 & 0 \\ \hline
w & 0 & 0 & \omega & 0 \\ \hline
t & 0 & 0 & 0 & 1 \\ \hline
\end{array}
\]
Here, $\omega$ denotes an element in $\kk$ satisfying $\omega^2=\zeta$.

As a summary, we get the following result.
\begin{theorem} \label{thm:K_8}
Let $\HH$ be a non-semisimple Hopf superalgebra of dimension $8$
such that neither $\HH$ nor $\HH^*$ is pointed and $\HH_{\bar{1}} \ne 0$.
Then $\HH$ is isomorphic to one of the eight Hopf superalgebras $\mathcal{K}_8(\zeta;\epsilon,\eta)$ $(\zeta\in\{\zeta_4,-\zeta_4\}, \epsilon,\eta\in\{0,1\})$,
which are pairwise non-isomorphic.
Moreover,
$\mathcal{K}_8(\zeta;0,0)$ and $\mathcal{K}_8(\zeta;1,1)$ are self-dual
and the dual of $\mathcal{K}_8(\zeta;0,1)$ is isomorphic to $\mathcal{K}_8(\zeta;1,0)$.
\end{theorem}

We say that a Hopf (resp.~super)algebra $\mathcal{K}$ has the {\it Chevalley property}
if the category of right $\mathcal{K}$-(resp.~super)comodules has the Chevalley property,
that is, the tensor product of any two simple objects is semisimple.
It follows directly that the Hopf superalgebras $\mathcal{K}_8(\zeta;\epsilon,\eta)$ have the Chevalley property
since the Hopf algebra $H_{16}((-1)^{\epsilon+\eta}\zeta)$ has the Chevalley property by \cite{CalDasMasMen04}.

\providecommand{\bysame}{\leavevmode\hbox to3em{\hrulefill}\thinspace}
\providecommand{\MR}{\relax\ifhmode\unskip\space\fi MR }
\providecommand{\MRhref}[2]{\href{http://www.ams.org/mathscinet-getitem?mr=#1}{#2}}
\providecommand{\href}[2]{#2}


\begin{thebibliography}{10}

\bibitem{AisMak14}
S.~Aissaoui and A.~Makhlouf, \emph{On classification of
  finite-dimensional superbialgebras and {H}opf superalgebras}, SIGMA Symmetry
  Integrability Geom. Methods Appl. \textbf{10} (2014), Paper 001, 24 pp.,
  \href{https://doi.org/10.3842/SIGMA.2014.001}{DOI 10.3842/SIGMA.2014.001}.
  \MR{3210634}

\bibitem{AndAngYam11}
N.~Andruskiewitsch, I.~Angiono, and H.~Yamane, \emph{On
  pointed {H}opf superalgebras}, New developments in {L}ie theory and its
  applications, Contemp. Math., vol.~544, Amer. Math. Soc., Providence, RI,
  2011, pp.~123--140,
  \href{https://doi.org/10.1090/conm/544/10752}{DOI 10.1090/conm/544/10752}.
  \MR{2849717}

\bibitem{AndEtiGel01}
N.~Andruskiewitsch, P.~Etingof, and S.~Gelaki, \emph{Triangular
  {H}opf algebras with the {C}hevalley property}, Michigan Math. J. \textbf{49}
  (2001), no.~2, 277--298,
  \href{https://doi.org/10.1307/mmj/1008719774}{DOI 10.1307/mmj/1008719774}.
  \MR{1852304}

\bibitem{BeaGar13}
M.~Beattie and G.~A.~Garc\'{\i}a, \emph{Classifying
  {H}opf algebras of a given dimension}, Hopf algebras and tensor categories,
  Contemp. Math., vol.~585, Amer. Math. Soc., Providence, RI, 2013,
  pp.~125--152,
  \href{https://doi.org/10.1090/conm/585/11615}{DOI 10.1090/conm/585/11615}.
  \MR{3077235}

\bibitem{CalDasMasMen04}
C.~C\u{a}linescu, S.~D\u{a}sc\u{a}lescu, A.~Masuoka, and C.~Menini,
  \emph{Quantum lines over non-cocommutative cosemisimple {H}opf algebras},
  J. Algebra \textbf{273} (2004), no.~2, 753--779,
  \href{https://doi.org/10.1016/j.jalgebra.2003.08.006}{DOI 10.1016/j.jalgebra.2003.08.006}.
  \MR{2037722}

\bibitem{Fuk97}
N.~Fukuda, \emph{Semisimple {H}opf algebras of dimension {$12$}}, Tsukuba
  J. Math. \textbf{21} (1997), no.~1, 43--54,
  \href{https://doi.org/10.21099/tkbjm/1496163160}{DOI 10.21099/tkbjm/1496163160}.
  \MR{1467220}

\bibitem{GarVay10}
G.~A.~Garc\'{\i}a and C.~Vay, \emph{Hopf algebras of
  dimension 16}, Algebr. Represent. Theory \textbf{13} (2010), no.~4, 383--405,
  \href{https://doi.org/10.1007/s10468-009-9128-7}{DOI 10.1007/s10468-009-9128-7}.
  \MR{2660853}

\bibitem{KacPal66}
G.~I.~Kac and V.~G.~Paljutkin, \emph{Finite ring groups} (Russian),
  Trudy Moskov. Mat. Ob\v{s}\v{c}. \textbf{15} (1966), 224--261.
  \MR{208401}

\bibitem{Kos77}
B.~Kostant, \emph{Graded manifolds, graded {L}ie theory, and
  prequantization}, Differential geometrical methods in mathematical physics
  (Proc. Sympos., Univ. Bonn, Bonn, 1975), Lecture Notes in Math., vol.~570,
  Springer, Berlin--New York, 1977, pp.~177--306,
  \href{https://doi.org/10.1007/BFb0087788}{DOI 10.1007/BFb0087788}.
  \MR{580292}

\bibitem{Maj94}
S.~Majid, \emph{Cross products by braided groups and bosonization},
  J. Algebra \textbf{163} (1994), no.~1, 165--190,
  \href{https://doi.org/10.1006/jabr.1994.1011}{DOI 10.1006/jabr.1994.1011}.
  \MR{1257312}

\bibitem{Mas95}
A.~Masuoka, \emph{Semisimple {H}opf algebras of dimension {$2p$}}, Comm.
  Algebra \textbf{23} (1995), no.~5, 1931--1940,
  \href{https://doi.org/10.1080/00927879508825319}{DOI 10.1080/00927879508825319}.
  \MR{1323710}

\bibitem{Mas95b}
A.~Masuoka, \emph{Semisimple {H}opf algebras of dimension {$6,8$}},
  Israel J. Math. \textbf{92} (1995), no.~1--3, 361--373,
  \href{https://doi.org/10.1007/BF02762089}{DOI 10.1007/BF02762089}.
  \MR{1357764}

\bibitem{Mas05}
A.~Masuoka, \emph{The fundamental correspondences in super affine groups and
  super formal groups}, J. Pure Appl. Algebra \textbf{202} (2005),
  no.~1--3, 284--312,
  \href{https://doi.org/10.1016/j.jpaa.2005.02.010}{DOI 10.1016/j.jpaa.2005.02.010}.
  \MR{2163412}

\bibitem{Mas08}
A.~Masuoka, \emph{Abelian and non-abelian second cohomologies of quantized
  enveloping algebras}, J. Algebra \textbf{320} (2008), no.~1, 1--47,
  \href{https://doi.org/10.1016/j.jalgebra.2008.03.034}{DOI 10.1016/j.jalgebra.2008.03.034}.
  \MR{2417975}

\bibitem{Mas12}
A.~Masuoka, \emph{Harish-{C}handra pairs for algebraic affine supergroup
  schemes over an arbitrary field}, Transform. Groups \textbf{17} (2012),
  no.~4, 1085--1121,
  \href{https://doi.org/10.1007/s00031-012-9203-8}{DOI 10.1007/s00031-012-9203-8}.
  \MR{3000482}

\bibitem{MasShi18}
A.~Masuoka and T.~Shibata, \emph{On functor points of affine supergroups},
  J. Algebra \textbf{503} (2018), 534--572,
  \href{https://doi.org/10.1016/j.jalgebra.2018.02.002}{DOI 10.1016/j.jalgebra.2018.02.002}.
  \MR{3780007}

\bibitem{Ng05}
S.-H.~Ng, \emph{Hopf algebras of dimension {$2p$}}, Proc. Amer. Math. Soc.
  \textbf{133} (2005), no.~8, 2237--2242,
  \href{https://doi.org/10.1090/S0002-9939-05-07804-4}{DOI 10.1090/S0002-9939-05-07804-4}.
  \MR{2138865}

\bibitem{Rad85}
D.~E.~Radford, \emph{The structure of {H}opf algebras with a projection},
  J. Algebra \textbf{92} (1985), no.~2, 322--347,
  \href{https://doi.org/10.1016/0021-8693(85)90124-3}{DOI 10.1016/0021-8693(85)90124-3}.
  \MR{778452}

\bibitem{SagVeg12}
D.~S.~Sage and M.~D.~Vega, \emph{Twisted {F}robenius--{S}chur indicators
  for {H}opf algebras}, J. Algebra \textbf{354} (2012), 136--147,
  \href{https://doi.org/10.1016/j.jalgebra.2011.12.026}{DOI 10.1016/j.jalgebra.2011.12.026}.
  \MR{2879228}

\bibitem{Shi23}
Y.~Shi and N.~Hu, \emph{Automorphism group of {Suzuki}'s {Hopf} algebra},
  Algebr. Represent. Theory \textbf{28} (2025), no.~4, 981--996,
  \href{https://doi.org/10.1007/s10468-025-10346-5}{DOI 10.1007/s10468-025-10346-5}.

\bibitem{Ste99}
D.~\c{S}tefan, \emph{Hopf algebras of low dimension},
  J. Algebra \textbf{211} (1999), no.~1, 343--361,
  \href{https://doi.org/10.1006/jabr.1998.7602}{DOI 10.1006/jabr.1998.7602}.
  \MR{1656583}

\bibitem{Zhu94}
Y.~Zhu, \emph{Hopf algebras of prime dimension}, Internat. Math. Res.
  Notices (1994), no.~1, 53--59,
  \href{https://doi.org/10.1155/S1073792894000073}{DOI 10.1155/S1073792894000073}.
  \MR{1255253}

\bibitem{ShiShiWak24}
T.~Shibata, K.~Shimizu, and R.~Wakao, \emph{Pointed {Hopf} superalgebras
  of dimension up to {$10$}}, Rev. Roumaine Math. Pures Appl.
  \textbf{69} (2024), no.~3--4, 539--573,
  \href{https://doi.org/10.59277/RRMPA.2024.539.573}{DOI 10.59277/RRMPA.2024.539.573}.
  Revised version: \href{https://arxiv.org/abs/2306.12637v3}{arXiv:2306.12637v3} (2026).

\end{thebibliography}
\end{document}